\documentclass[12pt]{extarticle}
\usepackage[margin=1in]{geometry}
\usepackage{authblk}
\usepackage[utf8]{inputenc}
\usepackage[greek,english]{babel}

\usepackage{amsmath}
\usepackage{hyperref}
\usepackage{mathrsfs}
\usepackage{amsfonts}
\usepackage{amssymb}
\usepackage{mathtools} 
\usepackage{graphicx}
\usepackage[dvipsnames]{xcolor}
\usepackage[autostyle]{csquotes}

\usepackage{soul} 
\usepackage[makeroom]{cancel}
\usepackage{caption}
\usepackage{subcaption}
\usepackage{latexsym,amsmath,amsfonts,amsthm,color,verbatim}
\usepackage{tikz}
\usepackage{tikz-3dplot}
\usetikzlibrary{calc}
\usepackage{bbm}
\usepackage[toc,page]{appendix}

\numberwithin{equation}{section}

\usetikzlibrary{arrows,decorations.markings}
\usetikzlibrary{decorations.pathmorphing,decorations.pathreplacing}

\usetikzlibrary{fadings}

\usepackage{calrsfs}
\usepackage{soul}
\usepackage{pifont}
\usepackage[normalem]{ulem}
\usepackage{hyperref}
\usepackage{enumerate}

\newcommand{\QQ}{\mathbb Q}
\newcommand{\NN}{\mathbb N}
\newcommand{\ZZ}{\mathbb Z}

\newcommand{\norm}[1]{\left\Vert #1\right\Vert}

\AtEndDocument{\bigskip{\footnotesize%
  \textsc{Department of Mathematics, University of Warwick } \par  
  \textit{E-mail address} : \texttt{ioannis.kousek@warwick.ac.uk} \par
}}

\newtheorem{theorem}{Theorem}[section]
\newtheorem{question}[theorem]{Question}
\newtheorem*{theorem'}{Theorem \ref{main ergodic theorem}}
\newtheorem*{theorem''}{Theorem \ref{finite fields p(x)}}
\newtheorem*{theorem'''}{Conjecture \ref{G-S statement}}
\newtheorem*{theorem*}{Theorem 1.13}
\newtheorem*{theorem**}{Theorem 1.14}
\newtheorem*{theorem***}{Conjecture 1.19}
\newtheorem*{theorem****}{Theorem 1.15}
\newtheorem{definition}[theorem]{Definition}

\newtheorem{remark}[theorem]{Remark}

\newtheorem{lemma}[theorem]{Lemma}
\newtheorem{proposition}[theorem]{Proposition}
\newtheorem{corollary}[theorem]{Corollary}

\newtheorem{conjecture}[theorem]{Conjecture}

\newtheorem*{remark*}{Remark}

\newcommand*\diff{\mathop{}\!\mathrm{d}}

\tikzset{->-/.style={decoration={
markings,
mark=at position .5 with { \arrow{>}}},postaction={decorate}}}

\tikzset{middlearrow/.style={
decoration={markings,
mark= at position 0.5 with { \arrow{#1}} ,
},
postaction={decorate}
}
}

\begin{document}

\author{Ioannis Kousek}
\date{}
\title{\textbf{Revisiting sums and products in countable and finite fields}}

\maketitle
\begin{abstract}
We establish a polynomial ergodic theorem for actions of the 
affine group of a countable field $K$. 
As an application, we deduce--via a variant of Furstenberg's 
correspondence principle--that for fields of characteristic 
zero, any ``large'' set $E\subset K$ contains 
``many'' patterns of the form $\{p(x)+y,xy\}$, for every 
non-constant polynomial $p(x)\in K[x]$. 

Our methods are flexible enough that they allow us to recover 
analogous density results in 
the setting of finite fields and, with the aid of a 
new finitistic variant of Bergelson's ``colouring trick'',
show that for $r\in \NN$ fixed, any $r-$colouring of a large 
enough finite field will contain monochromatic 
patterns of 
the form $\{x,p(x)+y,xy\}$. 

In a different direction, 
we obtain a double ergodic theorem
for actions of the affine group of a countable field. An 
adaptation of the argument for affine actions of finite fields 
leads to a generalisation of a theorem of Shkredov. Finally, 
to highlight the utility of the aforementioned finitistic 
``colouring trick'', we provide a conditional, elementary 
generalisation of Green and Sanders' $\{x,y,x+y,xy\}$ theorem.
\end{abstract}

\tableofcontents

\begin{center}
    \section{Introduction} \label{introduction}
\end{center}

\subsection{Historic background}
A well-known and still open question of Hindman 
(see, for example, \cite{4}) reads as follows. 

\begin{question} \label{Hindman's question}
Given any finite colouring of $\NN$, do there always exists $x,y\in \NN$ such that $\{x,y,x+y,xy\}$ is monochromatic, i.e. $x,y,x+y$ and $xy$ all have the same colour?     
\end{question}

In \cite{Mor}, Moreira proved the following
result marking significant progress towards 
an answer to Question \ref{Hindman's question}.

\begin{theorem}[Moreira] \label{Moreira} For any finite colouring of $\NN$ there exist (infinitely many) $x,y\in \NN$ such that $\{x,x+y,xy\}$ is monochromatic.
\end{theorem}

Prior to Moreira's theorem, Shkredov (\cite{12}) 
addressed its analogue for finite fields of prime order 
proving two density results. 

\begin{theorem}[Shkredov] \label{Shkredov weak} 
Let $\ZZ_p$ be a finite field of prime order $p$. If $A_1,A_2 \subset \ZZ_p$ are any sets with $|A_1||A_2|\geq 20p$, then there exist 
$x,y\in \ZZ^{*}_p:= \ZZ_p \setminus \{0\}$ such that $x+y\in A_1$ and $xy\in A_2$.
\end{theorem}

\begin{theorem}[Shkredov] \label{Shkredov main}
Let $\ZZ_p$ be a finite field of prime order $p$. If $A_1,A_2,A_3 \subset \ZZ_p$ are any sets with $|A_1||A_2||A_3|\geq 40p^{5/2}$, then there exist $x,y\in \ZZ^{*}_p$ such that $x+y\in A_1$, $xy\in A_2$ and $x\in A_3$.    
\end{theorem}

It follows from Theorem \ref{Shkredov main} that if 
$\ZZ_p$ is $r$-coloured and $p$ is large enough relative to $r$, then there exist $x,y\in \ZZ^{*}_p$ 
such that $\{x,x+y,xy\}$ is monochromatic. Later, the analogue of 
Question \ref{Hindman's question} for finite fields of 
prime order was solved by Green and Sanders in \cite{20} via the 
following quantitative result. 

\begin{theorem}[Green-Sanders] \label{Green-Sanders thm}
Let $r\in \NN$ be fixed and $\ZZ_p$ be a finite field of prime 
order $p$, with $p$ large enough. For any 
$r$-colouring of $\ZZ_p$ there 
are at least $c_r p^2$ monochromatic quadruples $\{x,y,x+y,xy\}$, 
where $c_r>0$ does not depend on $p$.  
\end{theorem}

Observe that Theorems \ref{Shkredov weak} and \ref{Shkredov main} are density results, while there
is no density version of the partition regularity 
Theorem \ref{Green-Sanders thm}. This was pointed out by 
Shkredov in \cite{12}. 

In the context of countable fields, Bowen and Sabok in \cite{BoSa} gave a positive answer to the analogue of Question 
\ref{Hindman's question}. By a compactness 
principle they also solved the analogue of this question 
for all finite fields as a corollary of their main theorem. 

Before that, Bergelson and Moreira in \cite{2} established the 
following analogue of Theorem \ref{Moreira} using 
methods from ergodic theory. 

\begin{theorem}[Bergelson-Moreira] \label{Moreira-Bergelson,1}
Let $K$ be a countable 
field and consider a finite colouring $K = \bigcup_{j=1}^r C_j$, 
$r\in \NN$. Then, there exists a colour $C_i$, $1\leq i \leq r$, 
and ``many'' $x, y \in K^{*}$, such that $\{x,x+y,xy\} \subset C_i.$    
\end{theorem}

In this setting, an appropriate notion of 
largeness, which guarantees patterns involving both addition and 
multiplication in 
any large set, turns out to be that of positive upper density 
with respect to double F\o lner sequences. We recall the 
definition given in \cite{2}.

\begin{definition} \label{Folner seq}
Let $K$ be a countable field. A double F\o lner sequence in $K$ is a sequence of (non-empty) finite subsets $(F_N)_{N\in \NN} \subset K$ which is asymptotically invariant under any fixed affine transformation of $K$, that is, 
$$\lim_{N\to \infty} \frac{\left| F_N \cap \left(x+F_N\right) \right|}{|F_N|}=\lim_{N\to \infty} \frac{\left| F_N \cap \left(xF_N\right) \right|}{|F_N|}=1,$$
for any $x\in K^{*}$. 
\end{definition}

This notion of sequence allows us to define asymptotic densities with good properties such as shift invariance. For a countable field $K$ and  $(F_N)_{N\in \NN}$ a double F\o lner sequence in $K$ as above, given a set $E\subset K$, its upper density with respect to $(F_N)_{N\in \NN}$ is defined as
$$\overline{\diff}_{(F_N)}(E)=\limsup_{N\to \infty} \frac{\left| E \cap F_N \right|}{|F_N|}.$$
Moreover, its lower density  with respect to $(F_N)_{N\in \NN}$ is defined as
$$\underline{\diff}_{(F_N)}(E)=\liminf_{N\to \infty} \frac{\left| E \cap F_N \right|}{|F_N|}$$
and whenever the limit exists we say that $E$ has a density 
with respect to $(F_N)_{N\in \NN}$ given by 
$\diff_{(F_N)}(E)=\overline{\diff}_{(F_N)}(E)=\underline{\diff}_{(F_N)}(E).$

Using a ``colouring trick" Bergelson and Moreira 
were able to recover Theorem 
\ref{Moreira-Bergelson,1} from essentially the following theorem, 
which we state vaguely. 

\begin{theorem}[Bergelson-Moreira] \label{main theorem density version B-M}
Let $K$ be a countable field, $(F_N)_{N\in \NN}$ be a 
double F\o lner sequence in $K$ and $E \subset K$ with 
$\overline{\diff}_{F_N}(E)>0$. Then, there exist ``many'' $x,y\in K$ 
such that $\{x+y,xy\} \subset E$.
\end{theorem}

An advantage of the statement of Theorem \ref{main theorem density version B-M}, over that of Theorem \ref{Moreira-Bergelson,1}, is that it's form can be handled with ergodic 
theoretic 
tools and methods. This is a general 
principle, discovered by Furstenberg in his 
seminal 
proof of Szemer\'edi's theorem (see \cite{8}). There 
he introduced a correspondence principle, which 
often allows one to translate a 
problem of finding patterns in large sets (subsets of 
the 
integers, of semi-groups, of fields, etc.) to a problem  
about recurrence in measure preserving systems.

The following ergodic theorem from \cite{2}, whose proof 
utilizes 
the group of affine transformations of a field $K$, 
defined as 
$\mathcal{A}_K:=\{f:x \mapsto ux+v \big| \ u,v\in K, u\neq 0 \}$, implies Theorem  \ref{main theorem density version B-M}.
We write $A_u$ for the map $x \mapsto x+u$, if $u\in K$ and $M_u$ for $x\mapsto ux$, if $u\in K^{*}:=K \setminus \{0\}$.
\begin{theorem}[Bergelson-Moreira] \label{Moreira-Bergelson,2}
Let $K$ be a countable field and $(F_N)_{N\in \NN}$ be a double 
F\o lner sequence in $K$. Let 
$(X,\mathcal{X},\mu)$ be a probability space on which we assume 
that $(T_g)_{g\in \mathcal{A}_K}$ acts by measure preserving 
transformations (m.p.t. for short).  Then, given any $B\in \mathcal{X}$, we have that 
$$\lim_{N\to \infty} \frac{1}{|F_N|} \sum_{u\in F_N} \mu( A_{-u}B \cap M_{1/u}B) \geq (\mu(B))^2.$$
\end{theorem}

\begin{remark*}
The fact that $(T_g)_{g\in \mathcal{A}_K}$ acts on 
$(X,\mathcal{X},\mu)$ by m.p.t. means that 
$(T_g)_{g\in \mathcal{A}_K}$ is a group action on $X$, so that 
$T_g \circ T_h = T_{g \circ h}$, any $g,h\in \mathcal{A}_K$, and 
that $\mu(A)=\mu(T_g^{-1}A),$ for any $A\in \mathcal{X}$ and $g\in \mathcal{A}_K$. Also, in an abuse of notation, we write $A_u$ for $T_{A_u}$ and $M_u$ for $T_{M_u}$, where 
$u\in K^{*}$.
\end{remark*}

\subsection{Main results}\label{main results}

A question which occurs naturally is whether we can extend Theorem
\ref{Moreira-Bergelson,1}, by finding monochromatic patterns of 
the form $\{x,p(x)+y,xy\}$, where $p(x)$ is a polynomial over $K$, 
other than $p(x)=x$. This is addressed by our first
main result (stated somewhat vaguely 
for now) which we formulate after an important--throughout this 
paper--definition.

\begin{definition}\label{admissible}
Given a field $K$ 
with prime characteristic $\text{char}(K)=q$, we say that a non-constant polynomial $p(x)\in K[x]$ is admissible for $K$, if $\deg(p(x))\leq q-1$. If $K$ is a countable field with $\text{char}(K)=0$, then any non-constant polynomial $p(x)\in K[x]$ is admissible for $K$.
\end{definition}

\begin{theorem} \label{main theorem partition version}
Let $K$ be a countable field and $p(x)\in K[x]\setminus K$ be any 
admissible polynomial. Then, for any finite colouring $K=C_1 \cup \dots \cup C_r,$ there exists a colour $C_j$, $1\leq j \leq r$, and ``many'' $x,y \in K^{*}$, so that $\{x,p(x)+y,xy\} \subset C_j.$
\end{theorem}

The density theorem which we will use to prove 
Theorem \ref{main theorem partition version} is the 
following. 

\begin{theorem} \label{main theorem density version}
Let $K$ be a countable field, $(F_N)_{N\in \NN}$ be a double F\o lner sequence in $K$ and $E \subset K$ with $\overline{\diff}_{F_N}(E)>0$. Then, for any admissible polynomial $p(x) \in K[x]\setminus K$
there exist ``many'' $x,y\in K$ such that $\{p(x)+y,xy\} \subset E$.
\end{theorem}

In the same spirit as in the end of Section $1.1$, 
Theorem \ref{main theorem density version} is implied 
by an ergodic theorem.

\begin{theorem} \label{main ergodic theorem}
Let $K$, $p(x) \in K[x]\setminus K$ and $(F_N)_{N\in \NN}$ be as in the 
statement of Theorem \ref{main theorem density version}. Let 
$(X,\mathcal{X},\mu)$ be a probability space on which we assume 
that $(T_g)_{g\in \mathcal{A}_K}$ acts by measure preserving 
transformations. Then, given any $f\in L^2(X,\mu)$ we have that
$$\lim_{N\to \infty} \frac{1}{|F_N|} \sum_{u\in F_N}  M_{u}A_{-p(u)}f = Pf,$$
where the limit is in $L^2$ and $P:L^2(X,\mu) \to L^2(X,\mu)$ 
denotes the orthogonal projection onto the subspace of $\mathcal{A}_K$-invariant functions.
\end{theorem}
 
\medskip

The proof of this statement is based on that of 
Bergelson and Moreira's proof of Theorem 
\ref{Moreira-Bergelson,2}, with additional applications
of van der Corput type of lemmas to facilitate an 
induction 
argument on the degree of the polynomial. This appears 
especially in the proof of the 
polynomial mean ergodic theorem of Proposition 
\ref{polynomial mean ergodic theorem }.

We also finitise the arguments used to 
prove Theorem \ref{main ergodic theorem} in order to recover the following analogue of 
our main density result, Theorem \ref{main theorem density version}, in the setting of finite fields. 

\begin{theorem} \label{finite fields p(x)}
Let $F$ be a finite field and let  
$p(x)\in F[x]$ be an admissible polynomial over $F$ of degree $q:=\deg(p(x))$. Then, if $E,G \subset F$ 
with $|E||G| > 2(q+2)|F|^{2-(1/2^{q-1})}$, there are $x,y\in F^{*}$, so that $xy\in E$ and $p(x)+y \in G$.    
\end{theorem}

In particular, letting $E=G$, we have the finite field version of 
the density statement that there exist $x,y\in F^{*}$ such that 
$\{p(x)+y,xy\} \subset E$, provided
$E \subset F$ is large enough. 

\medskip

We also produce a new finitistic version of 
the ``colouring trick'' mentioned earlier and with the aid of 
Theorem \ref{finite fields p(x)} recover the next partition 
regularity result.

\begin{theorem} \label{finite fields p(x) partition version}
Let $r,q\in \NN$ be fixed. Then, there exists $n(r,q)\in \NN$ with the following property. If 
$F$ is any finite field with $|F|\geq n(r,q)$ and $\text{char}(F)>q$ 
and $p(x)\in F[x]$ is a polynomial of $\deg(p(x))=q$, then for any 
finite colouring $F=C_1 \cup \dots \cup C_r$, there 
is a colour $C_j$ and $x,y\in F^{*}$, such that
$\{x,p(x)+y,xy\} \subset C_j.$    
\end{theorem}

\begin{remark*}
The assumption $\text{char}(F)>q$ is only to ensure that the 
polynomial $p(x) \in F[x]$ is admissible according to Definition \ref{admissible}.     
\end{remark*}

A special case of this theorem (when $p(x)=x$) is the 
partition 
regularity corollary of Shkredov's Theorem 
\ref{Shkredov main} mentioned after its statement. An 
advantage of the ergodic theoretic techniques used 
here is that we can recover more general polynomial 
patterns and also that the result holds for all finite 
fields and not only $\ZZ_p$. A perhaps more 
interesting feature, however, is the use of the novel--
in the finitistic setting--``colouring trick'', which, 
in a way, allows us to recover this partition 
regularity statement from a weaker density theorem.

\medskip

In a different direction we are also interested in the 
question of 
section $6.4$ of \cite{2}. Namely, is it true that under the 
assumptions of Theorem $\ref{Moreira-Bergelson,2}$ above we get triple 
intersections of the form 
$\mu(B \cap A_{-u}B \cap M_{1/u}B ) >0,$ for some 
$u \in K^{*}$? A generalization of the next non-commutative 
double ergodic theorem, without 
the assumption of ergodicity, would answer this question in the 
affirmative.

\begin{theorem} \label{main ergodic theorem 2} 
Let $K$ be a countable field and $(F_N)_{N\in \NN}$ be a double F\o lner sequence in $K$. Let 
$(X,\mathcal{X},\mu)$ be a probability space on which we assume 
that $(T_g)_{g\in \mathcal{A}_K}$ acts by measure preserving 
transformations and (crucially) we further assume that the action 
of the additive subgroup $S_A=\{A_u: u\in K\}$ is ergodic\footnote{The action $(T_g)_{g\in G}$
of a group $G$ on a probability space $(X,\mathcal{X},\mu)$ is ergodic if for any $A\in \mathcal{X}$ we have that $T_gA=A,\ \text{for all}\ g\in G \implies \mu(A)\in \{0,1\}$}. Then, 
given any $B \in \mathcal{X}$, we have that
$$\lim_{N\to \infty} \frac{1}{|F_N|} \sum_{u\in F_N} \mu(B \cap A_{-u}B \cap M_{1/u}B ) \geq (\mu(B))^3.$$
\end{theorem}

Unfortunately, we were unable to recover the result in its full generality. 
However, we make a natural conjecture.

\begin{conjecture} \label{conjecute countable fields}
In the context of Theorem \ref{main ergodic theorem 2}, if $S_A$ does not act ergodically, then given any $B \in \mathcal{X}$, we have that
$$\lim_{N\to \infty} \frac{1}{|F_N|} \sum_{u\in F_N} \mu(B \cap A_{-u}B \cap M_{1/u}B ) \geq (\mu(B))^4.$$
\end{conjecture}

In a relevant direction, Theorem \ref{Shkredov weak} 
was generalised to all finite fields, initially by 
Cilleruelo (\cite[Corollary $4.2$]{Cir}) and 
subsequently by Hanson (\cite[Theorem $1$]{Han}) and 
Bergelson and Moreira (\cite[Theorem $5.3$]{2}). 
However, a generalisation of Theorem \ref{Shkredov 
main} to any finite field remained open and we address 
this problem hereby through a ``finitisation'' of 
Theorem \ref{main ergodic theorem 2}. 

\begin{theorem}\label{main theorem Shkredov}
Let $F$ be any finite field and let $B_1,B_2,B_3 \subset F$ be any 
sets satisfying $|B_1||B_2||B_3|\geq 8|F|^{5/2}$. Then, there exist 
$x,y\in F^{*}$ such that $x+y\in B_1$, $xy\in B_2$ and $x\in B_3$. 
\end{theorem}

The ideas and techniques appearing in the proof of 
Theorem \ref{main ergodic theorem 2} spring from 
classical ergodic theoretic 
arguments used in proving multiple ergodic theorems. 
In this regard, the proof of Theorem \ref{main theorem
Shkredov}, which is more or less a ``finitisation'' of the 
above-mentioned proof, is different from 
Shkredov's original combinatorial proof of Theorem \ref{Shkredov main}.

Finally, by using the finitistic ``colouring trick'' 
and a finitistic version of Conjecture \ref{conjecute 
countable fields}, we provide an elementary, conditional 
proof of the following generalisation of Green and Sanders' 
Theorem \ref{Green-Sanders thm}.  

\begin{conjecture} \label{G-S statement}
Let $r\in \NN$ be fixed. Then, there is $n(r)\in \NN$, 
so that if $F$ is any finite field with $|F|\geq 
n(r)$ and $F=C_1\cup \dots \cup C_r$, there are $c_r|F|^2$ 
quadruples monochromatic $\{x,y,x+y,xy\}$, where $c_r>0$ does 
not depend on $|F|$. 
\end{conjecture}

\medskip

\textbf{Acknowledgments.} The author expresses gratitude to 
his advisor, Joel Moreira, for his guidance and beneficial 
discussions during the preparation of this paper. 
Thanks also go to Matt Bowen, Nikos Frantzikinakis and 
Andreas Mountakis for comments on earlier drafts.

\begin{center}
\section{Preliminaries and some useful results} \label{preliminaries}
\end{center}

\subsection{The action of the affine group}

For a countable field $K$, we denote by $\mathcal{A}_K=\{f:x \mapsto ux+v:\ u,v\in K,\ u\neq 0\}$ the group of affine transformations of $K$, with the operation of composition. The additive subgroup of $\mathcal{A}_K$ is denoted by $S_A$ and consists of the transformations $A_u: x  \mapsto x+u$, for $u\in K$. Similarly, the multiplicative subgroup, denoted by $S_M$, consists of transformations of the form $M_u: x \mapsto ux,$ for $u\in K^{*}$. 
The map $x \mapsto ux+v$ can be represented by the 
composition $A_vM_u$ and we have the trivial, but very useful throughout this paper, identity: 
\begin{equation} \label{M_uA_v}
M_uA_v=A_{uv}M_u.    
\end{equation}

The affine group appears naturally in our considerations because in 
order, for example, 
to find patterns $\{u+v,uv\}$ in a subset $E \subset K$ we can show 
that for some $u\in K^{*}$, the intersection $A_{-u}E \cap M_{1/u}E$ is non-empty.

We have already mentioned the utility of double F\o 
lner sequences as averaging schemes in $K$. The existence of 
such sequences was proved in Proposition $2.4$ of \cite{2}. 

\begin{proposition} \label{double folner definition}
Any countable field $K$ admits a sequence of non-empty finite sets 
$(F_N)_{N\in \NN}$ which forms a F\o lner sequence for both the 
actions of the additive group $(K,+)$ and the multiplicative group 
$(K^{*},\cdot)$. In other words, for any $u\in K^{*}$, we have that 
$$\lim_{N\to \infty} \frac{\left| F_N \cap \left(u+F_N\right) \right|}{|F_N|}=\lim_{N\to \infty} \frac{\left| F_N \cap \left(uF_N\right) \right|}{|F_N|}=1.$$
\end{proposition}

According to Lemma $2.6$ in \cite{2}, some 
transformations of double F\o lner sequences remain 
double F\o lner sequences. 

\begin{lemma} \label{transformations of Folner sequences}
Let $K$ be a countable field. If $(F_N)_{N\in \NN}$ is a double F\o lner sequence in $K$ and $b \in K^{*}$, then $(bF_N)_{N\in \NN}$ is still a double F\o lner sequence in $K$.
\end{lemma}

We will further consider a probability space $(X,\mathcal{X},\mu)$ and a measure preserving action $(T_g)_{g\in \mathcal{A}_K}$ of $\mathcal{A}_K$ on $X$. In this context, we denote $L^2(X,\mu)$ by $H$ and let $(U_g)_{g\in \mathcal{A}_K}$ be given by $(U_gf)(x)=f(T_g^{-1}x)$, for $x\in X$ and $f\in H$. This is known as the unitary Koopman representation of $\mathcal{A}_K$. Abusing notation we will usually write $A_uf$ instead of $U_{A_u}f$ and $M_uf$ instead of $U_{M_u}f$. By $P_A$ we denote the orthogonal projection from $H$ onto the subspace of vectors which are fixed by the action of the additive subgroup $S_A$. Also, by $P_M$ we denote the orthogonal projection from $H$ onto the subspace of vectors fixed under the action of $S_M$.   

The useful and unintuitive fact that the projections $P_A$ 
and $P_M$ commute was established in Lemma $3.1$ of \cite{2}. 

\begin{lemma} \label{P_AP_M commute}
For any $f\in H$ we have that 
$$P_A P_Mf = P_M P_A f.$$ 
\end{lemma}

By Lemma \ref{P_AP_M commute} we see that $P_AP_Mf$ is 
invariant under the actions of both $S_A$ and $S_M$ and that $P_AP_Mf$ is an orthogonal projection. Since the subgroups $S_A$ and $S_M$ generate the whole group $\mathcal{A}_K$, it follows that $P=P_AP_M=P_MP_A$ is the orthogonal projection from $H$ onto the subspace of vectors fixed under the action of $\mathcal{A}_K$. 

\bigskip

\subsection{Ergodic theorems and van der Corput lemmas}

The mean ergodic theorem for unitary 
representations of countable abelian
groups, which we will extend later for our purposes, has the following form and a proof of 
this version can be found for example in \cite{14}, Theorem $5.4$.

\begin{theorem} \label{mean ergodic theorem}
Let $G$ be a countable abelian group and $(F_N)_{N\in \NN}$ be a 
F\o lner sequence in $G$. Let also $H$ be a Hilbert space and $(U_g)_{g\in G}$ be a unitary representation of $G$ on $H$. Then for any $f\in H$, 
$$\lim_{N\to \infty} \frac{1}{|F_N|} \sum_{g\in F_N} U_gf = Pf,$$
where the limit is in the strong topology of $H$ and $P$ denotes the orthogonal projection onto the subspace
of vectors fixed under $G$.
\end{theorem}

\begin{remark*}
One may consider for example the cases where, provided that 
$\mathcal{A}_K$ acts by m.p.t. on a probability space $(X,\mathcal{X},\mu)$, we have that $H=L^2(X,\mu)$, 
$G=S_A$ or $G=S_M$ and then $P=P_A$ or $P=P_M$, respectively.     
\end{remark*}

We will consider an adaptation of 
the van der Corput 
lemma for unitary representations of countable abelian groups. A proof--of a stronger version--appears 
in Theorem $2.12$ of \cite{Mor vdC}.

\begin{lemma} \label{generalized vdC}
Let $(G,\cdot)$ be a countable abelian group and $(a_u)_{u\in G}$ be a 
bounded sequence of vectors in a Hilbert space $H$, indexed by the 
elements of $G$. Let $(F_N)_{N\in \NN}$ be a F\o lner sequence in $G$. If 
$$\lim_{M\to \infty} \frac{1}{|F_M|} \sum_{v\in F_M} \limsup_{N\to \infty} \frac{1}{|F_N|} \left| \sum_{u\in F_N} \langle a_{u \cdot v}, a_u\rangle \right| =0,$$
then also 
$$\lim_{N\to \infty} \frac{1}{|F_N|} \sum_{u\in F_N} a_u =0. $$
\end{lemma}

\begin{remark*}
This, in particular, holds when $(G,\cdot)=(K,+)$ or when $(G,\cdot)=(K^{*},\cdot)$ for some countable field $K$ and $(F_N)_{N\in \NN}$ is a double F\o lner sequence in $K$.     
\end{remark*}

Another version of the van der Corput lemma, 
which will be used in Section 
\ref{section main ergodic theorem 2}, follows as a
corollary of the inequality given
in Lemma $1$, Chapter $21$ of Host and Kra's book \cite{HK}.  

\begin{proposition} \label{vdc new}
Let $(G,\cdot)$ be a countable abelian group with identity $1$ and for each $b\in G$ let $(a_u(b))_{u\in G}$ be a 
bounded sequence of vectors in a Hilbert space $H$ with norm $\norm{\cdot}$, indexed by the 
elements of $G$. Let $(F_N)_{N\in \NN}$ be a F\o lner sequence in $G$. If for all $d\neq 1$,
$$ \lim_{M\to \infty} \frac{1}{|F_M|} \sum_{b\in F_M} \limsup_{N\to \infty} \frac{1}{|F_N|} \sum_{u\in F_N} \langle a_{u \cdot d}(b), a_u(b)\rangle =0,$$
then also
$$\lim_{M\to \infty} \frac{1}{|F_M|} \sum_{b\in F_M} \limsup_{N\to \infty} \norm{\frac{1}{|F_N|} \sum_{u\in F_N} a_u(b)}^2=0. $$ 
\end{proposition}

For finite groups, a version of the van 
der Corput lemma is given by the following simple 
equality. We will use this to adapt our infinite 
ergodic theorems to the setting of finite fields. 

\begin{proposition}\label{finite vdC}
Let $(G,\cdot)$ be a finite group and $(f(g))_{g\in G}$ be a sequence 
taking values in a Hilbert space $H$. 
Then, 
$$\norm{ \sum_{g\in G} f(g)}^2 = \sum_{g\in G} \sum_{h\in G} \langle f(g \cdot h), f(g)\rangle.$$
\end{proposition}

Finally, we shall find the next classical result useful. 

\begin{lemma}\label{cesaro convergence of a_u and a_u^2 to 0}
Let $(a_u)_{u\in G}$ be a bounded, non-negative sequence,
indexed by 
elements of a countable (amenable) group $G$ and let $(G_N)_{N\in \NN}$ 
be a F\o lner sequence in $G$. Then
$$ \lim_{N\to \infty} \frac{1}{|G_N|} \sum_{u\in G_N} a_u = 0 \iff \lim_{N\to \infty} \frac{1}{|G_N|} \sum_{u\in G_N} a_u^2 = 0.$$
\end{lemma}

\begin{center}
\section{Proofs of Theorems \ref{main theorem density version} and \ref{main ergodic theorem} } \label{polynomial infinite fields}
\end{center}

Throughout this section we assume that $K$ is a countable field, 
$(F_N)_{N\in \NN}$ is a double F\o lner sequence in $K$ and $p(x) \in 
K[x]$ is a non-constant admissible polynomial over $K$, according to Definition \ref{admissible}. We also let $(X,\mathcal{X},\mu)$ be a 
probability space on which we assume that $(T_g)_{g\in \mathcal{A}_K}$ 
acts by measure preserving 
transformations. In consistency with the notation from Section 
\ref{preliminaries}, $H=L^2(X,\mu)$, $P: H \to H$ denotes the orthogonal 
projection from $H$ onto the subspace of functions fixed under the action 
of $\mathcal{A}_K$ and $P_A$, $P_M$ are the orthogonal projections on the 
subspaces of vectors fixed under the additive action $S_A$ and the 
multiplicative action $S_M$, respectively. Moreover, $(U_g)_{g\in \mathcal{A}_K}$ is the unitary Koopman representation of $\mathcal{A}_K$ (for details recall the discussion after Lemma \ref{transformations of Folner sequences}). Again, for simplicity, we will write $A_u$ instead of $U_{A_u}$ and $M_u$ instead of $U_{M_u}$. 

Before embarking on the proof of Theorem \ref{main ergodic theorem} 
we show the ensuing, straightforward corollary of it. 

\begin{corollary} \label{main ergodic theorem corollary proof}
If $K$, $p(x)\in K[x]\setminus K$, $(F_N)_{N\in \NN}$ and 
$(X,\mathcal{X},\mu)$ are as above, then for any $B\in \mathcal{X}$, 
we have that 
$$\lim_{N\to \infty} \frac{1}{|F_N|} \sum_{u\in F_N} \mu( A_{-p(u)}B \cap M_{1/u}B ) \geq (\mu(B))^2.$$
\end{corollary}

\begin{proof}
For $B\in \mathcal{X}$ we see that 
$$\lim_{N\to \infty} \frac{1}{|F_N|} \sum_{u\in F_N} \mu( A_{-p(u)}B \cap M_{1/u}B ) =\lim_{N\to \infty} \frac{1}{|F_N|} \sum_{u\in F_N} \int_X A_{-p(u)}\mathbbm{1}_B \cdot M_{1/u} \mathbbm{1}_B\ d\mu,$$ 
which can be written as (using that $M_u$ is preserves $\mu$, for all $u\in K^{*}$)
\begin{equation} \label{45}
\lim_{N\to \infty} \frac{1}{|F_N|} \sum_{u\in F_N} \int_X M_u A_{-p(u)}\mathbbm{1}_B \cdot \mathbbm{1}_B\ d\mu.
\end{equation}
By Theorem \ref{main ergodic theorem} applied for $f=\mathbbm{1}_B$, \eqref{45} becomes 
$$ \int_X (P\mathbbm{1}_B) \cdot \mathbbm{1}_B\ d\mu \geq (\mu(B))^2.$$
For the last inequality observe that $P$ is an orthogonal projection and so 
$$\int_X (P\mathbbm{1}_B) \cdot \mathbbm{1}_B\ d\mu = \int_X (P\mathbbm{1}_B)^2\ d\mu \geq \left( \int_X P\mathbbm{1}_B\ d\mu  \right)^2, $$
by the Cauchy-Schwarz inequality. Finally, because $P1=1$ we have that 
$$\int_X P\mathbbm{1}_B\ d\mu=\int_X \mathbbm{1}_B\ d\mu=\mu(B)$$ and thus we conclude.
\end{proof}

\begin{remark*}
A similar argument shows that if in the context of 
Theorem \ref{main ergodic theorem} the action 
of $\mathcal{A}_K$ is also ergodic, then for any $B,C\in 
\mathcal{X}$ we have that
$$\lim_{N\to \infty} \frac{1}{|F_N|} \sum_{u\in F_N} 
\mu( A_{-p(u)}B \cap M_{1/u}C ) \geq \mu(B) \mu(C).$$
\end{remark*}

\medskip

For the special case $p(x)=x$, the proof of Theorem
\ref{main ergodic theorem} was given in \cite{2}. We only mention 
that in the proof of the linear case in \cite{2}, the authors 
relied on a version of the mean ergodic Theorem
\ref{mean ergodic theorem} for the action of $S_A$. For the 
polynomial case of Theorem \ref{main ergodic theorem} we will use 
the subsequent generalization, which is a polynomial mean ergodic 
theorem for the action of $S_A$. For that we will need an 
application of the van 
der Corput trick utilizing the 
additive structure of $K$, which facilitates an 
induction argument on the polynomial's degree. 

\begin{theorem} \label{polynomial mean ergodic theorem }
Let $K$ be a countable field and 
$p(x) \in K[x]\setminus K$ be admissible. Let also  
$(F_N)_{N\in \NN}$ be a double F\o lner sequence in $K$ and 
$(X,\mathcal{X},\mu)$ a probability space, on which $(T_{A_u})_{u\in K}$ acts by measure preserving transformations (see also the beginning of this
section). Then, given any $f\in H$ we have that
$$\lim_{N\to \infty} \frac{1}{|F_N|} \sum_{u\in F_N}  A_{p(u)}f = P_Af,$$
where the limit is in the strong topology of $H$.    
\end{theorem}

\begin{proof}
We prove the case $\text{char}(K)=q$, some $q\in \mathbbm{P}$ (see 
also Remark \ref{char=0}). If 
$p(x)=ax+b$, where $a,b\in K$ and $a\neq 0$, then it follows by the 
mean ergodic theorem that
$$\lim_{N\to \infty} \frac{1}{|F_N|}\sum_{u\in F_N} A_{au+b}f = \lim_{N\to \infty} \frac{1}{|F_N|}\sum_{u\in aF_N+b} A_{u}f =P_Af.$$
Note that here we used the fact that 
$(aF_N+b)_{N\in \NN}$ is still a F\o lner sequence for the additive group $(K,+)$, in view of Lemma \ref{transformations of Folner sequences} and the obvious observation that shifts of F\o lner sequences are also F\o lner sequences in any group. Now, 
assume the statement holds for polynomials of degree $m-1$, where $2 \leq m \leq q-1$ and let $p(x) \in K[x]\setminus K$ have degree $m$, i.e., $p(x)=q_0+q_1x+\dots+q_mx^m$, $q_0,\dots,q_m \in K$ and $q_m \neq 0$. First, we let $f\in H$ be such that $P_Af=0$ and set $a_u=A_{p(u)}f,$ $u \in K$. Then, for any $b\in K^{*}$, we have that 
$$\langle a_{u+b}, a_u \rangle=\langle A_{p(u+b)-p(u)}f , f \rangle.$$ 
Observe that 
$$p(u+b)-p(u)= q_m \sum_{k=0}^{m-1} \binom{m}{k} 
u^k \cdot b^{m-k} +r_b(u),$$
where $\deg(r_b(x)) \leq m-2$. Therefore, 
$$p(u+b)-p(u)= m \cdot (q_m b) u^{m-1} +r'_b(u),$$
where $\deg(r_b'(x)) \leq m-2$, and since $q_mb \neq 0$, the above argument
shows that the polynomial $g_{b}(x)=p(x+b)-p(x)$ has degree $m-1$ in $K[x]$. 

We note that an issue arises 
in allowing the polynomial's degree to be $q$, in which case if, for example, $p(x)=x^q$, 
then $g_{b}(x)=b^q$ is a constant, because $(x+b)^q=x^q+b^q$ in 
a field of characteristic $q$.  

Returning to the proof, by the induction hypothesis and the 
assumption on $f$, we see that for any $b\neq 0$,
$$\lim_{N\to \infty} \frac{1}{|F_N|} \sum_{u\in F_N} \langle a_{u+b} , a_u \rangle=\lim_{N\to \infty} \frac{1}{|F_N|} \sum_{u\in F_N} \langle A_{g_{b}(u)}f , f \rangle = \langle P_Af,f \rangle =0.$$
Thus, an application of the van der Corput trick as in Lemma \ref{generalized vdC} gives us that 
$$\lim_{N\to \infty} \frac{1}{|F_N|} \sum_{u\in F_N} A_{p(u)}f =0,$$
in $H$, when $P_Af=0$. Finally, for a general $f\in H$ we can write $f=P_Af+(f-P_Af)$ and from the above and linearity it follows that 
$$\lim_{N\to \infty} \frac{1}{|F_N|} \sum_{u\in F_N}  A_{p(u)}f = \lim_{N\to \infty} \frac{1}{|F_N|} \sum_{u\in F_N}  A_{p(u)}P_Af =P_Af.$$
\end{proof}

\begin{remark} \label{char=0}
Note that the same proof in the case of $\text{char(K)}=0$ (for 
example when $K=\QQ$), gives 
the same result for polynomials of arbitrarily large degree, 
because then it always holds that $x \mapsto p(x+b)-p(x)$ is a 
polynomial of degree equal to $\deg(p(x))-1$, when $b\neq 0$.    
\end{remark}

We will now give the proof of Theorem \ref{main ergodic theorem}, 
the statement of which we recall for the reader's convenience. 

\begin{theorem*}
Let $K$, $(F_N)_{N\in \NN}$, $p(x) \in K[x]\setminus K$,   $(X,\mathcal{X},\mu)$ and $(T_g)_{g\in \mathcal{A}_K}$ be as in the beginning of this section. Then, given any $f\in H$ we have that
$$\lim_{N\to \infty} \frac{1}{|F_N|} \sum_{u\in F_N}  M_{u}A_{-p(u)}f = Pf,$$
where the limit is in the strong topology of $H$.     
\end{theorem*}

\begin{proof}
Let $f\in H$ and assume that $P_Af=0$. For $u\in K^{*}$ we 
now set $a_u=M_uA_{-p(u)}f$ and then, for any $b\in K^{*}$ we 
have that 
$$ \langle a_{ub} , a_u \rangle = \langle A_{-p(ub)+p(u)/b}f , M_{1/b}f \rangle. $$
If $p(x)=q_0+q_1x+\dots+q_mx^m$, $q_0,\dots,q_m \in K$ and $q_m \neq 0$ ($m<q$ if $\text{char}(K)=q$), then 
$$p(ub)-p(u)/b=q_0\frac{b-1}{b}+u\left(q_1\frac{b^2-1}{b}\right)+\dots+u^m\left(q_m\frac{b^{m+1}-1}{b} \right),$$
which, for $b\notin \{0,1,-1\}$ fixed, is also a polynomial of degree $m$.
Thus, applying Theorem \ref{polynomial mean ergodic theorem } we have that for $b\notin \{0,1,-1\}$,
$$\lim_{N\to \infty} \frac{1}{|F_N|} \sum_{u\in F_N}  \langle a_{ub} , a_u \rangle = \langle P_Af , M_{1/b}f\rangle=0.$$
Once again, the van der Corput lemma implies that for $P_Af=0$, 
$$\lim_{N\to \infty}  \frac{1}{|F_N|} \sum_{u\in F_N} M_uA_{-p(u)}f = 0,$$
and this allows us to conclude just like in the case of Theorem \ref{polynomial mean ergodic theorem }, after decomposing a general $f\in H$ as $f=P_Af+(f-P_Af)$.
\end{proof}

Using some quantitative bounds for the set of return 
times, which can be extracted from the proof of 
Corollary \ref{main ergodic theorem corollary proof}, 
and the 
variant of Furstenberg's correspondence 
principle established in Theorem $2.8$ of \cite{2}, we 
can recover the following precise version of Theorem 
\ref{main theorem density version}. The proof is a 
straightforward adaptation of the proof of Theorem 
$2.5$ from Theorem $2.10$ in \cite{2}, which amounts 
to the special case that $p(x)=x$.

\begin{theorem} \label{main theorem density version proof}
Let $K$ be a countable field, $p(x) \in K[x]\setminus K$ an admissible polynomial and 
$(F_N)_{N\in \NN}$ be a double F\o lner sequence in $K$. Let $E \subset K$ with $\overline{\diff}_{(F_N)}(E)>0.$ Then, for any $\epsilon>0$ we have that 
$$\underline{\diff}_{(F_N)}\left( \{u\in K^{*}: \overline{\diff}_{(F_N)}\left( (E-p(u)) \cap (E/u)  \right) \geq (\overline{\diff}_{(F_N)}(E))^2-\epsilon \} \right) >0.$$
In less precise terms, for each element of a large set of $u \in K^{*}$ there is a large set of $v\in K^{*}$ satisfying $\{ v+p(u),vu \} \subset E$. 
\end{theorem}

To conclude the results of this section we give a precise 
statement of Theorem \ref{main theorem partition version}.

\begin{theorem} \label{main theorem partition version proof}
Let $K$ be a countable field, $(F_N)_{N\in \NN}$ a double F\o lner sequence in $K$ and $p(x)\in K[x]\setminus K$ an admissible polynomial. Then, for any finite colouring 
$K=C_1 \cup \dots \cup C_r$,
there exists a colour $C_j$ such that 
$$\overline{\diff}_{(F_N)}\left( \{u\in C_j: \overline{\diff}_{(F_N)}\left( \{ v\in K: \{u,p(u)+v,uv\} \subset C_j \} \right) \} \right) >0.$$
\end{theorem}    
The proof of Theorem 
\ref{main theorem partition version proof} is based on the 
``colouring trick'' of (and is almost identical to) the 
proof of Theorem $4.1$ in \cite{2}, and 
therefore is omitted. The only difference being that we rely 
on Corollary \ref{main ergodic theorem corollary proof}, 
while in \cite{2} the authors relied on its special case of a 
linear polynomial. 

It seems like our methods are not rigid enough to deal with 
non-admissible polynomials according to Definition 
\ref{admissible} because of the comment in the proof
of Theorem \ref{polynomial mean ergodic theorem }, so we make 
the following natural questions. 

\begin{question}\label{Q1}
Does Corollary \ref{main ergodic theorem corollary proof} hold 
if $p(x) \in K[x]$ is not admissible? 
\end{question}

\begin{question} \label{Q2}
Does Theorem \ref{main theorem partition version proof} (or a 
vague version as in Theorem 
\ref{main theorem partition version}) hold for non-admissible 
polynomials $p(x) \in K[x]$?
\end{question}

We note that a positive answer to Question \ref{Q1} would also 
imply a positive answer to Question \ref{Q2} by the same 
argument that is used for the case of admissible polynomials.

\begin{center}
\section{A finite fields version of Theorem \ref{main theorem density version}} \label{finite fields p(x) section}
\end{center}

In this section we will adapt the proof of Theorem \ref{main theorem density version} to the finite fields setting and prove Theorem \ref{finite fields p(x)}.

For a finite field $F$ we consider its group of affine transformations,  
$\mathcal{A}_{F}$, which consists of the maps of the form $x \mapsto ux+v,$ 
where $u\in F^{*}$ and $v\in F$. We also let $(X,\mathcal{X},\mu)$ be a probability space on which $\mathcal{A}_{F}$ acts by measure preserving 
transformations, with $(T_g)_{g\in \mathcal{A}_{F}}$ denoting the action. As before, we let $S_A=\{A_u : u\in F\}$, where 
$A_u(x)=x+u$ and $S_M=\{M_u: u\in F^{*}\}$, where $M_u(x)=xu$. Also, in an abuse of notation, if 
$(U_g)_{g\in \mathcal{A}_{F}}$ is the Koopman 
representation of $\mathcal{A}_{F}$ on $L^2(X,\mu)$ we write $A_u$ for $U_{A_u}$ and $M_u$ for $U_{M_u}$, where for example, for $f\in L^2(X,\mu)$ we have that $U_{A_u}f(x)=f(T_{A_u}^{-1}x)=f(T_{A_{-u}}x)$.  

Moreover, if $P_A$ is the orthogonal projection onto the space of 
functions invariant under the subgroup $S_A$, we see that
$P_Af(x)=\frac{1}{|F|} \sum_{u\in F} A_uf(x)$ and if $P_M$ is 
the projection onto the space of functions invariant 
under $S_M$, then $P_Mf(x)=\frac{1}{|F^{*}|} \sum_{u\in F^{*}} M_uf(x)$. 
We will begin with a finitistic version of the 
polynomial mean ergodic theorem of Section \ref{polynomial infinite fields} and then prove an analogue of 
Theorem \ref{main ergodic theorem}. As in the infinite case, $P_A$ and $P_M$ exhibit commuting behavior (see the proof of Theorem $5.1$ in \cite{2}).

\medskip

\begin{proposition} \label{P_AP_M finite}
For $f\in L^2(X,\mu)$ and $P_A$, $P_M$ as above, we have that $P_AP_Mf=P_MP_Af$. 
\end{proposition}

Thus, $P_AP_M$ is an orthogonal projection onto the 
subspace of functions invariant under $\mathcal{A}_F$.
The promised finitistic analogue of Theorem \ref{polynomial mean ergodic theorem } is this. 

\begin{proposition} \label{finite PET}
Let $F$ be a finite field and assume that $\mathcal{A}_{F}$ acts on $(X,\mathcal{X},\mu)$ as in the beginning of this section. Let also $p(x) \in F[x]\setminus F$ be an admissible polynomial of degree $q:=\deg(p(x))$. Then, for any $f\in L^2(X,\mu)$ we have that
$$\norm{\frac{1}{|F|} \sum_{u\in F} A_{p(u)}f - P_Af}_2^2 \leq \frac{q-1}{|F|^{1/2^{q-2}}}\norm{f-P_Af}_2^2.$$
\end{proposition}

\begin{proof}
If $p(x)=ax+b$, $a,b\in F$ and $a\neq 0$, this is obvious, for $p(F)=\{au+b: u\in F\}=F$, whence it is enough to make a change of variables and use the definition of $P_A$. Assume now that the conclusion holds for polynomials of degree at most $q<r-1$ and let $p(x) \in F[x]$ be a polynomial of degree $q+1\leq r-1$, where $\text{char}(F)=r$, some $r\in \mathbbm{P}$. Then, 
$$\frac{1}{|F|} \sum_{u\in F} A_{p(u)}f =\frac{1}{|F|} \sum_{u\in F} A_{p(u)}P_Af + \frac{1}{|F|} \sum_{u\in F} A_{p(u)}\Tilde{f},$$
where $\Tilde{f}=f-P_Af$, so that $P_A\Tilde{f}=0.$ Clearly, 
$$\frac{1}{|F|} \sum_{u\in F} A_{p(u)}P_Af = P_Af.$$
On the other hand, by Proposition \ref{finite vdC} it follows that  
\begin{equation} \label{38}
\norm{\frac{1}{|F|} \sum_{u\in F} A_{p(u)}\Tilde{f}}_2^2 = \frac{1}{|F|} \sum_{v\in F} \frac{1}{|F|} \sum_{u\in F} \langle A_{p(u+v)-p(u)}\Tilde{f} ,\Tilde{f} \rangle.
\end{equation}
Since $\deg(p(x))=q+1\leq r-1$, the polynomial $p(x+v)-p(x)$ has degree $q$ for any $v\neq 0$ (this would no longer be true if the degree of $p(x)$ was $r$ just like the infinite field case), and since $P_A\Tilde{f}=0$, the induction hypothesis implies that 
\begin{equation} \label{39}
\norm{\frac{1}{|F|} \sum_{u\in F} A_{p(u+v)-p(u)}\Tilde{f}}_2^2 \leq \frac{q-1}{|F|^{1/2^{q-2}}}\norm{\Tilde{f}}_2^2.
\end{equation}
Finally, we see that 
$$\frac{1}{|F|} \sum_{v\in F} \frac{1}{|F|} \sum_{u\in F} \langle A_{p(u+v)-p(u)}\Tilde{f} ,\Tilde{f} \rangle \leq 
\frac{1}{|F|}\norm{\Tilde{f}}_2^2 + \frac{1}{|F|} 
\sum_{v\in F^{*}} \frac{1}{|F|} \sum_{u\in F} 
\langle A_{p(u+v)-p(u)}\Tilde{f} ,\Tilde{f} \rangle,$$
which, by an application of the Cauchy-Schwarz inequality is bounded above by
\begin{equation} \label{39'}
\frac{1}{|F|}\norm{\Tilde{f}}_2^2 + \norm{\frac{1}{|F|} \sum_{u\in F} A_{p(u+v)-p(u)}\Tilde{f}}_2\norm{\Tilde{f}}_2.
\end{equation}
Using \eqref{39} in \eqref{39'} and then by \eqref{38} it follows that
$$\norm{\frac{1}{|F|} \sum_{u\in F} A_{p(u)}\Tilde{f}}_2^2 \leq
\frac{1}{|F|} \norm{\Tilde{f}}_2^2 +  \frac{\sqrt{q-1}}{|F|^{1/2^{q-1}}}\norm{\Tilde{f}}_2^2 \leq \frac{q}{|F|^{1/2^{q-1}}}\norm{\Tilde{f}}_2^2.$$
\end{proof}

We isolate the following estimate that appears in the proof of 
the finitistic analogue of Corollary \ref{main ergodic theorem corollary proof}, that is, Theorem \ref{finite fields recurrence} 
below. This estimate is the finitistic analogue of 
Theorem \ref{main ergodic theorem} for functions orthogonal to the 
space of functions fixed under the action of $S_A$.

\begin{proposition}\label{estimate for p(x)}
Let $F$ be a finite field and assume that $\mathcal{A}_F$ acts on 
$(X,\mathcal{X},\mu)$ as in the beginning of this section. Let also 
$p(x) \in F[x]\setminus F$ be an admissible polynomial of degree 
$q:=\deg(p(x))$. Let $f=\mathbbm{1}_C - P_A\mathbbm{1}_C$ for some 
$C\in \mathcal{X}$. Then, 
\begin{equation} \label{25}
\norm{\frac{1}{|F^{*}|} \sum_{u\in F^{*}} M_uA_{-p(u)}f}_2^2 < 2(q+2)\mu(C)/|F^{*}|^{1/2^{q-1}}. 
\end{equation}
\end{proposition}

\begin{proof}
From Proposition \ref{finite vdC} we have that 
\begin{multline} \label{20}
\norm{\frac{1}{|F^{*}|} \sum_{u\in F^{*}} M_uA_{-p(u)}f}_2^2 = \frac{1}{|F^{*}|} \sum_{u\in F^{*}} \frac{1}{|F^{*}|} \sum_{v\in F^{*}} \langle M_{uv}A_{-p(uv)}f , M_uA_{-p(u)f} \rangle = \\
\frac{1}{|F^{*}|} \sum_{v\in F^{*}} \frac{1}{|F^{*}|} \sum_{u\in F^{*}} \langle A_{-p(uv)+p(u)/v}f , M_{1/v}f \rangle.
\end{multline}
Now, for $v= \pm -1$ (in fact for any $v\in F^{*}$, but this wouldn't lead to a practically useful bound) it is easy to see that 
\begin{equation} \label{19'}
\frac{1}{|F^{*}|} \sum_{u\in F^{*}} \langle A_{-p(uv)+p(u)/v}f , M_{1/v}f \rangle \leq \norm{f}_2^2.    
\end{equation}
On the other hand, for any $v\in F^{*},$ $v\neq \pm 1$, we have
\begin{equation} \label{16}
\left| \frac{1}{|F^{*}|} \sum_{u\in F^{*}} \langle A_{-p(uv)+p(u)/v}f , M_{1/v}f \rangle \right| \leq \norm{ \frac{1}{|F^{*}|} \sum_{u\in F^{*}} A_{-p(uv)+p(u)/v}f}_2 \norm{f}_2. 
\end{equation}
Moreover, 
\begin{multline} \label{17}
\norm{ \frac{1}{|F^{*}|} \sum_{u\in F^{*}} A_{-p(uv)+p(u)/v}f}_2 \leq \\ 
\norm{ \frac{|F|}{|F^{*}|}\frac{1}{|F|} \sum_{u\in F} A_{-p(uv)+p(u)/v}f}_2 + \norm{ \frac{1}{|F^{*}|} A_{-p(0)+p(0)/v}f}_2. 
\end{multline}
But, if $v\not\in \{0,1,-1\}$, then $-p(uv)+p(u)/v$ is a polynomial of same degree as $p(u)$, and so by Proposition \ref{finite PET} and because $P_Af=0$, \eqref{17} becomes\footnote{We used that $|F| \big/ |F^{*}|  \left( \sqrt{q-1} \big/ |F|^{1/2^{q-1}} \right)+1 \big/ |F^{*}| \leq q \big/ |F^{*}|^{1/2^{q-1}}$, whenever $|F|\geq 3$.} 
\begin{equation*} \label{18}
\norm{ \frac{1}{|F^{*}|} \sum_{u\in F^{*}} A_{-p(uv)+p(u)/v}f}_2 \leq \frac{q}{|F^{*}|^{1/2^{q-1}}}\norm{f}_2.
\end{equation*}
Using this in \eqref{16} we get that (for $v\notin \{0,1,-1\})$  
\begin{equation} \label{19}
\frac{1}{|F^{*}|} \sum_{u\in F^{*}} \langle A_{-p(uv)+p(u)/v}f , M_{1/v}f \rangle  \leq \frac{q}{|F^{*}|^{1/2^{q-1}}}\norm{f}_2^2. 
\end{equation}
Combining \eqref{19'} and \eqref{19} it follows from \eqref{20} 
that 
\begin{equation*}
\norm{\frac{1}{|F^{*}|} \sum_{u\in F^{*}} M_uA_{-p(u)}f}_2^2 \leq (q+2)\norm{f}_2^2/|F^{*}|^{1/2^{q-1}}.
\end{equation*}
It is shown in the proof of Theorem $5.1$ in \cite{2} that 
$\norm{f}_2 \leq \sqrt{2\mu(C)}$. Therefore, the latter inequality 
readily implies \eqref{25} and so we conclude.
\end{proof}

\begin{theorem} \label{finite fields recurrence}
Let $F$ be a finite field and assume that $\mathcal{A}_F$ acts on $(X,\mathcal{X},\mu)$ as in the beginning of this section. Let also $p(x) \in F[x]\setminus F$ be an admissible polynomial of degree $q:=\deg(p(x))$.
Then, for any set $B\in \mathcal{X}$, such 
that $(\mu(B))^2> 2(q+2) \big/ 
|F^{*}|^{1/2^{q-1}}$, there exists $u\in F^{*}$ so that $\mu(B\cap M_uA_{-p(u)}B)>0.$    

If, in addition, the action of $S_A$ is ergodic, then 
for any sets $B,C \in \mathcal{X}$ which satisfy 
$\mu(B)\mu(C)> 2(q+2) \big/ |F^{*}|^{1/2^{q-1}}$, there is some $u\in F^{*}$ with
$\mu(B\cap M_uA_{-p(u)}C)>0.$
\end{theorem}

\begin{remark*}
For the case $p(x)=x$, that is, when $q=1$, the bounds in this 
statement coincide with those that Bergelson and Moreira found in 
\cite{2}. 
\end{remark*}

\begin{proof}
Let $B,C \in \mathcal{X}$. For the second conclusion it suffices to 
prove the following averages are positive (for the first conclusion 
we prove the same thing with $B=C$)
\begin{align} 
\frac{1}{|F^{*}|} \sum_{u\in F^{*}} \mu(B\cap M_u A_{-p(u)}C)\ = \ \nonumber & \langle  \mathbbm{1}_B , \frac{1}{|F^{*}|} \sum_{u\in F^{*}} M_uA_{-p(u)}\mathbbm{1}_C \rangle = \\
\langle \mathbbm{1}_B , \frac{1}{|F^{*}|} \sum_{u\in F^{*}} M_uA_{-p(u)}P_A\mathbbm{1}_C \rangle \  + & 
\ \langle \mathbbm{1}_B , \frac{1}{|F^{*}|} 
\sum_{u\in F^{*}} M_uA_{-p(u)}f \rangle, \label{14}
\end{align}
where $f=\mathbbm{1}_C-P_A\mathbbm{1}_C$. Now, we observe that 
\begin{equation}\label{11}
\langle \mathbbm{1}_B , \frac{1}{|F^{*}|} \sum_{u\in 
F^{*}} M_uA_{-p(u)}P_A\mathbbm{1}_C \rangle = \langle \mathbbm{1}_B , \frac{1}{|F^{*}|} \sum_{u\in F^{*}} M_uP_A\mathbbm{1}_C \rangle = \langle \mathbbm{1}_B , P_MP_A \mathbbm{1}_C \rangle. 
\end{equation}
If $S_A$ acts ergodically, then $P_A\mathbbm{1}_C=\mu(C)$ and so \eqref{11} becomes 
\begin{equation} \label{12}
\langle \mathbbm{1}_B , \frac{1}{|F^{*}|} \sum_{u\in F^{*}} M_uA_{-p(u)}P_A\mathbbm{1}_C \rangle = \mu(B) \mu(C).    
\end{equation}
If $B=C$ and we don't assume ergodicity, then 
$P_MP_A\mathbbm{1}_B=P\mathbbm{1}_B$, where $P$ is the projection 
onto the space of functions invariant under $\mathcal{A}_F$ by 
Proposition \ref{P_AP_M finite}. Therefore $P1=1$ and it follows by 
the Cauchy-Schwarz inequality that 
\begin{equation} \label{13}
\langle \mathbbm{1}_B , \frac{1}{|F^{*}|} \sum_{u\in F^{*}} M_uA_{-p(u)}P_A\mathbbm{1}_B \rangle =
\langle \mathbbm{1}_B , P\mathbbm{1}_B \rangle = \norm{P\mathbbm{1}_B}_2^2 \geq (\mu(B))^2.
\end{equation}
For the last averages in \eqref{14} another application of Cauchy-Schwarz's inequality gives that 
\begin{equation} \label{15}
\left| \langle \mathbbm{1}_B , \frac{1}{|F^{*}|} \sum_{u\in F^{*}} M_uA_{-p(u)}f \rangle \right| \leq \sqrt{\mu(B)} \norm{\frac{1}{|F^{*}|} \sum_{u\in F^{*}} M_uA_{-p(u)}f}_2.
\end{equation}
So, from \eqref{25} in Proposition \ref{estimate for p(x)} the inequality in \eqref{15} now becomes 
\begin{equation*} \label{21}
\left| \langle \mathbbm{1}_B , \frac{1}{|F^{*}|} \sum_{u\in F^{*}} M_uA_{-p(u)}f \rangle \right| \leq \sqrt{2(q+2)\mu(B)\mu(C)} \big/ |F^{*}|^{1/2^{q}}.
\end{equation*}
In conclusion, \eqref{14} implies that
\begin{equation}\label{22}
\frac{1}{|F^{*}|} \sum_{u\in F^{*}} \mu(B\cap M_uA_{-p(u)}C) \geq \langle \mathbbm{1}_B , P_MP_A\mathbbm{1}_C \rangle - \sqrt{2(q+2)\mu(B)\mu(C)} \big/ |F^{*}|^{1/2^{q}}. 
\end{equation} 
As we have alluded to in the beginning of this proof, there are now two routs. If $S_A$ acts ergodically, then \eqref{22} becomes 
\begin{equation}\label{28} 
\frac{1}{|F^{*}|} \sum_{u\in F^{*}} \mu(B\cap M_uA_{-p(u)}C) \geq \mu(B)\mu(C) - \sqrt{2(q+2)\mu(B)\mu(C)} \big/ |F^{*}|^{1/2^{q}},
\end{equation} 
and this is positive whenever $\mu(B)\mu(C) > 2(q+2) \big/ |F^{*}|^{1/2^{q-1}} $. If we don't assume ergodicity and $B=C$, then we have 
\begin{equation}\label{24}
\frac{1}{|F^{*}|} \sum_{u\in F^{*}} \mu(B\cap M_uA_{-p(u)}B) \geq (\mu(B))^2 - \sqrt{2(q+2)}\mu(B) \big/ |F^{*}|^{1/2^{q}},
\end{equation} 
which is positive precisely when $(\mu(B))^2> 2(q+2) \big/ |F^{*}|^{1/2^{q-1}}.$
\end{proof}

Some quantitative bounds for the set of return times 
in the previous theorem--which will be used in the proof of Theorem 
\ref{finite fields p(x)} given below and in Section \ref{partition for finite fields}--are the following. 

\begin{corollary} \label{estimates 1}
Let $F$ be a finite field and assume that $\mathcal{A}_F$ acts on $(X,\mathcal{X},\mu)$ by m.p.t. Let also $p(x) \in F[x]\setminus F$ be an admissible polynomial of degree $q:=\deg(p(x))$, $B \in \mathcal{X}$ and $\delta < \mu(B)$. Then, the set of return times $D:=\{u\in F^{*}: \mu(B \cap M_uA_{-p(u)}B)>\delta \}$ satisfies 
\begin{equation} \label{26} 
\frac{|D|}{|F^{*}|} \geq \frac{(\mu(B))^2 - \sqrt{2(q+2)}\mu(B) \big/ |F^{*}|^{1/2^{q}} - \delta}{\mu(B)}.    
\end{equation}   
If, in addition, the action of $S_A$ is ergodic, then for any $B,C \in \mathcal{X}$ and $\delta < \min{\{\mu(B),\mu(C)\}}$, the set $D':=\{u\in F^{*}: \mu(B \cap M_uA_{-p(u)}C)>\delta \}$ satisfies 
\begin{equation} \label{27} 
\frac{|D'|}{|F^{*}|} \geq  \frac{\mu(B)\mu(C) - \sqrt{2(q+2)\mu(B)\mu(C)} \big/ |F^{*}|^{1/2^{q}} - \delta}{\min{\{\mu(B),\mu(C)\}}}.   
\end{equation}   
\end{corollary}

\begin{proof}
By \eqref{24} we know that
\begin{equation*}
\frac{1}{|F^{*}|} \sum_{u\in F^{*}} \mu(B\cap M_uA_{-p(u)}B) \geq (\mu(B))^2 - \sqrt{2(q+2)}\mu(B) \big/ |F^{*}|^{1/2^{q}}.
\end{equation*}     
At the same time, $\mu(B \cap M_uA_{-p(u)}B) \leq \mu(B)$ implies 
that
$$\frac{1}{|F^{*}|} \sum_{u\in F^{*}} \mu(B \cap M_uA_{-p(u)}B) \leq \frac{|D|}{|F^{*}|}\mu(B) +  \left(1-\frac{|D|}{|F^{*}|} \right)\delta=\delta + \frac{|D|}{|F^{*}|}(\mu(B)-\delta).$$ 
Combining the two inequalities we see that 
$$(\mu(B))^2 - \sqrt{2(q+2)}\mu(B) \big/ |F^{*}|^{1/2^{q}} \leq \delta + \frac{|D|}{|F^{*}|}(\mu(B)-\delta) $$
and thus
$$\frac{|D|}{|F^{*}|}\mu(B) \geq (\mu(B))^2 - \sqrt{2(q+2)}\mu(B) \big/ |F^{*}|^{1/2^{q}} - \delta,$$
which is \eqref{26}. For the ergodic case we use \eqref{28} instead of \eqref{24} and the rest is similar.
\end{proof}
    
We shall conclude this section by proving Theorem \ref{finite fields p(x)}.

\begin{theorem**} 
Let $F$ be a finite field. Then, if 
$p(x)\in F[x]$
is an admissible polynomial over $F$ of degree $q:=\deg(p(x))$ and $E,G \subset F$ 
with $|E||G| > 2(q+2)|F|^{2-(1/2^{q-1})}$, there are $u,v\in F^{*}$, so that $vu \in E$ and $p(u)+v \in G$.  
\end{theorem**}

\begin{remark*}
To give a better taste of the bounds, if 
we are looking for patterns of the form $\{uv , u+v^2\}$ in a 
subset $E$ of a field of order $3^6=729$, then our method 
demands that $|E|>2\sqrt{2 \ 3^{9}} \approx 396$, and for a 
field of order $3^7=2187$, that $|E|>2\sqrt{2} \ 3^{21/4} \approx 904$. 
\end{remark*}

\begin{proof}
Consider the action by affine transformations of $\mathcal{A}_{F}$ 
on $F$ with the normalised counting measure $\mu$, 
i.e. $\mu(B)=|B|/|F|$, for any $B\subset F$. Then the action of 
$S_A$ is ergodic. Now, for $s<\min{\{|E|,|G|\}}$, we let $\delta=s/|F|$ and  
$D:=\{u\in F^{*}: \mu(E \cap M_uA_{-p(u)}G)>\delta\}.$ By Corollary \ref{estimates 1} we know that 
\begin{equation*}
\frac{|D|}{|F^{*}|} \geq  \frac{\mu(E)\mu(G) - \sqrt{2(q+2)\mu(E)\mu(G)} \big/ |F^{*}|^{1/2^{q}} - \delta}{\min{\{\mu(E),\mu(G)\}}}.   
\end{equation*}    
This means that
\begin{equation} \label{29} 
|D| \geq \frac{|E||G||F^{*}|/|F| - |F^{*}|^{1-1/2^q}\sqrt{2(q+2)|E||G|}- s|F^{*}|}{\min{\{|E|,|G|\}}}.   
\end{equation}
Observe that for $u\in D$ we have that 
$$\frac{s}{|F|}=\delta \leq \mu(E \cap M_uA_{-p(u)}G) = \frac{\left| M_{1/u}E \cap A_{-p(u)}G \right|}{|F|},$$
which means that for each $u\in D$ there are $s$ elements $v\in F$, such that $vu\in E$ and $v+p(u)\in G$. 
\end{proof}
 
\begin{center}
\section{A new ``colouring trick'' and partition regularity for
finite fields} \label{partition for finite fields}
\end{center}

In this section we will adapt the infinite ``colouring 
trick'' presented in Section $4$ of \cite{2} in order to 
recover a partition regularity result for finite fields, 
namely Theorem \ref{finite fields p(x) partition version}, 
from weaker density results established in Section 
\ref{finite fields p(x) section}; essentially from the proof 
of Theorem \ref{finite fields p(x)}. We recall Theorem \ref{finite fields p(x) partition version} for convenience.

\begin{theorem****}
Let $r,q\in \NN$ be fixed. Then, there is $n(r,q) \in \NN$, so 
that for a finite field $F$ with $|F|\geq n(r,q)$ and 
$\text{char}(F)>q$ and a polynomial $p(x)\in F[x]$ of 
$\deg(p(x))=q$, any colouring $F=C_1 \cup \dots \cup C_r$ 
contains monochromatic triples of the form $\{u,p(u)+v,uv\}$.  
\end{theorem****}

\begin{proof}
Let $r\in \NN$, $r>1$, be fixed and let $F$ be any finite field with $|F| \geq n(r,q)$, for $n(r,q)$ to be determined later. For an $r$-colouring of such a field, we can permute the colours if necessary and assume that $|C_1| \geq |C_2| \geq \dots \geq |C_r|$. Clearly then, $|C_1| \geq |F|\big/ r$. Next, we pick a number $1\leq r' \leq r$ in the following manner. If $|C_2| < |F|\big/r^4$, we set $r'=1$. Else, we have that $|C_2| \geq |F|\big/r^4$ and $r'\geq 2$. 
Then, we either have that $|C_3| \geq |F|\big/r^8$, whence $r'\geq 2$ or not and 
let $r'=2$. In this fashion we set 
$$r':= \max{\Big\{1 \leq j \leq r: |C_1| \geq |F|\big/r\ , \ |C_2| \geq |F|\big/r^4\ ,\ \dots \ , \ |C_{j}| \geq |F|\big/r^{2^{j}} \Big\}}.$$ 
Let $C=C_1 \times \dots \times C_{r'}$. We consider the natural measure 
preserving action of $\mathcal{A}_{F}$ on $F^{r'}$ (defined coordinate-wise), 
with the counting measure $\nu$ given by $\nu(E)=|E|/|F^{r'}|$, for any 
$E \subset F^{r'}$. For any $\delta=s \big/ |F^{*}| < \nu(C)$, let 
\begin{equation*} 
D=\{u\in F^{*}: \nu(C \cap M_uA_{-p(u)}C) > \delta \},
\end{equation*}
the size of which we can bound below by Corollary \ref{estimates 1}, which implies that
\begin{equation} \label{55} 
|D| \geq \frac{(\nu(C))^2|F^{*}| - \nu(C)\sqrt{2(q+2)}|F^{*}|^{1-1/2^q} - s}{\nu(C)}.
\end{equation}
Next, we show that
\begin{equation} \label{56}
|D| > |F|-(|C_1|+\dots+|C_{r'}|)=|C_{r'+1}|+\dots+|C_r|.
\end{equation}
Observe that by the definition of $r'$ it follows that
\begin{equation} \label{57}
|C_{r'+1}|+\dots+|C_r| \leq  (r-r')|F|\big/ r^{2^{(r'+1)}} < |F|\big/ r^{2^{(r'+1)}-1}.
\end{equation} 
Combining \eqref{55} with \eqref{57}, we see that \eqref{56} follows from 
$$\nu(C)|F^{*}| -\sqrt{2(q+2)}|F^{*}|^{1-1/2^q} - s / \nu(C) > |F|\big/ r^{2^{(r'+1)}-1},$$
or equivalently that, 
\begin{equation} \label{58}
\nu(C) > \sqrt{2(q+2)} \big/ |F^{*}|^{1/2^q}  + 1\big/ r^{2^{(r'+1)}-1} + s \big/ \left(|F^{*}|\nu(C) \right) + 1 \big/ \left(|F^{*}|  r^{2^{(r'+1)}-1}  \right). 
\end{equation}
Using the definition of $C$ and $r'$ it holds that 
$$\nu(C)=\frac{|C_1|\cdots |C_{r'}|}{|F^{r'}|} \geq \frac{1}{r} \cdot \frac{1}{r^4}\cdot \frac{1}{r^8} \cdots \frac{1}{r^{2^{r'}}}=\frac{1}{r^{(1+4+8+\dots+2^{r'})} }.$$
Now, one can see that\footnote{For $r'\geq 2$ we have that $2^{r'+1}-\left(2^{r'}+\dots+2^2 \right)=4$}
$$\frac{1}{r^{(1+4+\dots+2^{r'})} }-\frac{1}{r^{2^{(r'+1)}-1}} = \frac{r^{2^{(r'+1)}-1-(2^{r'}+\dots+2^2+1)}-1}{r^{2^{(r'+1)}-1}} = \frac{r^2-1}{r^{2^{(r'+1)}-1}},$$
when $r'\geq 2$. If $r'=1$, then the equation becomes $1\big/ r-1\big/ r^3 = \left( r^2-1 \right) \big/ r^3$.
Finally, \eqref{58} follows from 
\begin{equation} \label{64} 
\frac{r^2-1}{r^{2^{(r'+1)}-1}} \geq  \sqrt{2(q+2)} \big/ |F^{*}|^{1/2^q} +  s \big/ \left(|F^{*}|\nu(C) \right) + 1 \big/ \left(|F^{*}|  r^{2^{(r'+1)}-1}  \right), 
\end{equation}
which holds for $|F|\geq n(r,q)$, with $n(r,q)$ large 
enough, since the RHS goes to $0$ as $|F| \to \infty$, for $r,q$ fixed. By \eqref{56} we know 
that $D \cap \left(C_1 \cup \dots \cup C_{r'}\right) \neq \emptyset$ as 
$$\left| D \cap \left(C_1 \cup \dots \cup C_{r'}\right) \right| \geq |D|-|C_{r'+1}|-\dots-|C_r|.$$
Thus, there must exist $u\in C_1 \cup \dots \cup C_{r'}$, such that $\nu(C \cap M_uA_{-p(u)}C) > s \big/ |F^{*}|$. Then, if $u \in C_j$, for $1 \leq j \leq r'$, by the definition of $C$ and the measure $\nu$ we will also have that 
\begin{equation} \label{65}
\frac{| C_j/u \cap \left(  C_j-p(u)\right) |}{|F|}=\mu(C_j \cap M_uA_{-p(u)} C_j) > \frac{s}{|F^{*}|} > \frac{s}{|F|}    
\end{equation} and hence $C_j/u \cap (C_j-p(u)) \neq \emptyset$. This implies the existence of $u,v\in F$ with $u\neq 0$ such that $\{u,p(u)+v,uv\} \subset C_j$. In particular, for each $u\in D \cap \left(C_1 \cup \dots \cup C_{r'}\right)$ there are, by \eqref{65}, at least $s$ monochromatic triples $\{u,p(u)+v,uv\}$.   
\end{proof}

\begin{remark}
The observant reader will have noticed that the proof above 
actually gives that 
$$\left| D \cap \left(C_1 \cup \dots \cup C_{r'}\right) \right| \geq |F^{*}| \left( \frac{r^2-1}{r^{2^{(r'+1)}-1}} -  \frac{\sqrt{2(q+2)}}{|F^{*}|^{1/2^q}}  - \frac{s}{ |F^{*}|\nu(C) } - \frac{1}{|F^{*}|  r^{2^{(r'+1)}-1}} \right).$$
Therefore, for any finite field with $|F^{*}|\geq n(r,q)$ we see that
$$\left| D \cap \left(C_1 \cup \dots \cup C_{r'}\right) \right| \geq c_{r,q} \cdot |F|,$$
where, whenever $n(r,q)$ is large enough,
$$c_{r,q}=  \frac{r^2-1}{r^{2^{(r'+1)}-1}} -  \frac{\sqrt{2(q+2)}}{n(r,q)^{1/2^q}}  - \frac{s}{ n(r,q)\cdot \nu(C) } - \frac{1}{n(r,q)\cdot  r^{2^{(r'+1)}-1}} >0$$
is a constant that does not depend on $|F|$.
Using the concluding comments of the previous proof, as $s=\delta |F^{*}|$ we have a total of $c'_{r,q} |F|^2$ monochromatic triples of the form $\{u,u+v,uv\}$, where $c'_{r,q}>0$ is a constant that does not depend on $|F|$.
\end{remark}

\begin{center}
\section{Proof of Theorem \ref{main ergodic theorem 2}} \label{section main ergodic theorem 2}
\end{center}

Throughout this short section we will assume that $K$ is a countable field and 
$(F_N)_{N\in \NN}$ is a double F\o lner sequence in $K$. 
We also let $(T_g)_{g\in \mathcal{A}_K}$ denote an action of $\mathcal{A}_K$ on some probability space 
$(X,\mathcal{X},\mu)$ by measure preserving 
transformations. For reference, our main goal is to prove the next 
result, part of which was initially stated as 
Theorem \ref{main ergodic theorem 2}.

\medskip

\begin{theorem} \label{main ergodic theorem 2 proof}
Let $K$, $(F_N)_{N\in \NN}$, $(X,\mathcal{X},\mu)$ and 
$(T_g)_{g\in \mathcal{A}_K}$ be as above. Also, we (crucially) further assume that the action 
of the additive subgroup $S_A=\{A_u: u\in K\}$ is ergodic. Then, 
given any $B \in \mathcal{X}$, we have that
$$\lim_{N\to \infty} \frac{1}{|F_N|} \sum_{u\in F_N} \mu(B \cap A_{-u}B \cap M_{1/u}B) \geq (\mu(B))^3.$$
If, in addition, the action of $S_M$ is ergodic, then for any $B_1,B_2,B_3 \in \mathcal{X}$ we have that 
$$\lim_{N\to \infty} \frac{1}{|F_N|} \sum_{u\in F_N} \mu(B_1 \cap A_{-u}B_2 \cap M_{1/u}B_3) \geq \mu(B_1)\mu(B_2)\mu(B_3).$$
\end{theorem}

The proof is based on the following (double) ergodic theorem.

\begin{theorem} \label{P_Af}
Let $K$, $(F_N)_{N\in \NN}$, $(X,\mathcal{X},\mu)$ and 
$(T_g)_{g\in \mathcal{A}_K}$ be as in the beginning of this section. We further assume that the 
action of the additive subgroup $S_A$ is ergodic. Then, for any $f,g \in L^{\infty}(X,\mu)$ we have that
$$ \lim_{N\to \infty} \frac{1}{|F_N|} \sum_{u\in F_N} M_uA_{-u}f \cdot M_{u}g =P_M g \cdot P_Af, $$
where the limit is in $L^2$.
\end{theorem}

\begin{proof}
Without loss of generality we assume that $f$ and $g$ are 
real-valued functions. We begin by decomposing $f$ as 
$f=P_Af+\Tilde{f}$, 
where $\Tilde{f}=f-P_Af$.
Then, 
\begin{equation} \label{70} 
\frac{1}{|F_N|} \sum_{u\in F_N} M_uA_{-u}f \cdot M_{u}g = \frac{1}{|F_N|} \sum_{u\in F_N} M_uA_{-u}P_Af \cdot M_{u}g + \frac{1}{|F_N|} \sum_{u\in F_N} M_uA_{-u}\Tilde{f} \cdot M_{u}g.      
\end{equation}
As $P_Af$ is a constant by the ergodicity of $S_A$, it follows by (the ergodic) Theorem \ref{mean ergodic theorem} that 
$$\lim_{N\to \infty}  \frac{1}{|F_N|} \sum_{u\in F_N} M_uA_{-u}P_Af \cdot M_{u}g = P_Mg \cdot P_Af.$$
Hence, the proof will follow from \eqref{70} if we can show that 
$$\lim_{N\to \infty}  \frac{1}{|F_N|} \sum_{u\in F_N} M_uA_{-u}\Tilde{f} \cdot M_{u}g = 0.$$
To this end, we let $a_u=M_uA_{-u}\Tilde{f} \cdot M_{u}g$, for $u\in K^{*}$. By the van der 
Corput trick (see Lemma \ref{generalized vdC}) for $(K^{*},\cdot)$ it suffices to show that 
\begin{equation} \label{71}
\lim_{M\to \infty} \frac{1}{|F_M|} \sum_{b\in F_M} \limsup_{N\to \infty} \left| \frac{1}{|F_N|} \sum_{u\in F_N} \langle a_{ub},a_u \rangle \right| =0.
\end{equation}
To this end we note that for $b\neq 0$,
\begin{align*}
\langle a_{ub} , a_u \rangle = \langle M_{ub}A_{-ub}\Tilde{f} \cdot M_{ub}g , M_uA_{-u}\Tilde{f} \cdot M_{u}g \rangle & = \\ 
\langle M_{b}A_{-ub}\Tilde{f} \cdot M_{b}g , A_{-u}\Tilde{f} \cdot g \rangle & = \int_X g \cdot M_{b}g \cdot M_bA_{-ub}\Tilde{f}  \cdot  A_{-u}\Tilde{f}\ d\mu,  
\end{align*}
where we have used that $M_v$ preserves $\mu$. Hence, using the equality $M_uA_v=A_{uv}M_u$ (see \ref{M_uA_v}), for all $u,v\in K^{*}$, we have
$$ \frac{1}{|F_N|} \sum_{u\in F_N} \langle a_{ub},a_u \rangle = \frac{1}{|F_N|} \sum_{u\in F_N} \int_X g \cdot M_{b}g \cdot A_{-ub^2}M_b\Tilde{f}  \cdot  A_{-u}\Tilde{f}\ d\mu $$
and so it suffices to show that 
\begin{equation}\label{1}
\lim_{M\to \infty} \frac{1}{|F_M|} \sum_{b\in F_M} \limsup_{N\to \infty} \left| \frac{1}{|F_N|} \sum_{u\in F_N} \int_X g\cdot M_bg \cdot A_{-u}\Tilde{f}  \cdot  A_{-ub^2}M_b\Tilde{f} \right| =0.
\end{equation}
By Cauchy-Schwarz's inequality and Lemma \ref{cesaro convergence of a_u and a_u^2 to 0} the convergence in \eqref{1} follows from 
\begin{equation*}
\lim_{M\to \infty} \frac{1}{|F_M|} \sum_{b\in F_M} \limsup_{N\to \infty} \norm{\frac{1}{|F_N|} \sum_{u\in F_N} A_{-u}\Tilde{f}  \cdot  A_{-ub^2}M_b\Tilde{f}}^2_2=0.
\end{equation*}
Now, using Proposition \ref{vdc new} with $(G,\cdot)=(K,+)$ and $a_u(b) = A_{-u}\Tilde{f}  \cdot  A_{-ub^2}M_b\Tilde{f}$, for any $u,b\in K$, $b\neq 0$, we reduce this to showing that 
\begin{equation}\label{1'}
\lim_{M\to \infty} \frac{1}{|F_M|} \sum_{b\in F_M} \limsup_{N\to \infty} \frac{1}{|F_N|} \sum_{u\in F_N} \langle a_{u+d}(b),a_u(b) \rangle = 0,
\end{equation}
for any $d\neq 0$.
As before we see that 
$$ \langle a_{u+d}(b),a_u(b) \rangle = \int_X A_{u(b^2-1)-d}\Tilde{f} \cdot A_{-db^2}M_b\Tilde{f} \cdot A_{u(b^2-1)}\Tilde{f} \cdot M_{b}\Tilde{f}\ d\mu.$$
Now, since $A_{u(b^2-1)-d}\Tilde{f} \cdot A_{u(b^2-1)}\Tilde{f}=A_{u(b^2-1)}\left( \Tilde{f} \cdot  A_{-d}\Tilde{f} \right)$ and for $b\notin \{-1,1\}$, $p(x)=(b^2-1)x$ is a polynomial of degree $1$ in $K[x]$, we may use the mean ergodic Theorem \ref{mean ergodic theorem} to obtain that the averages in \eqref{1'} become 
\begin{equation} \label{2}
\lim_{M\to \infty} \frac{1}{|F_M|} \sum_{b\in F_M} \int_X 
 P_A( \Tilde{f}\cdot  A_{-d}\Tilde{f} ) \cdot A_{-db^2}M_{b}\Tilde{f} \cdot M_{b}\Tilde{f}\ d\mu .
\end{equation}
As $S_A$ is ergodic, the projection $ P_A( \Tilde{f}\cdot  A_{-d}\Tilde{f})$ is a constant and so, using \eqref{M_uA_v} and the invariance of $\mu$ under $M_v$ once again, \eqref{2} becomes 
\begin{equation} \label{2'}
\lim_{M\to \infty} \frac{1}{|F_M|} \sum_{b\in F_M}   P_A( \Tilde{f}\cdot  A_{-d}\Tilde{f}  ) \int_X 
 A_{-db}\Tilde{f} \cdot \Tilde{f}\ d\mu. 
\end{equation}
Because $(F_M)_{M\in \NN}$ is a double F\o lner sequence in $K$ and $d\neq 0$ it follows by Proposition \ref{transformations of Folner sequences} and the mean ergodic theorem that
$$\lim_{M\to \infty} \frac{1}{|F_M|} \sum_{b\in F_M}   \int_X A_{-db}\Tilde{f} \cdot \Tilde{f}\ d\mu =\int_X P_A\Tilde{f}\cdot \Tilde{f}\ d\mu=0,$$
by the definition of $\Tilde{f}$. Therefore, the limit in \eqref{2'} equals zero and so \eqref{71} follows.
\end{proof}

From Theorem \ref{P_Af} we can readily recover Theorem \ref{main ergodic theorem 2 proof}.

\begin{proof}[Proof of Theorem \ref{main ergodic theorem 2 proof}]
For $B\in \mathcal{X}$ we see that 
\begin{align*}
\lim_{N\to \infty} \frac{1}{|F_N|} \sum_{u\in F_N} \mu(B\cap A_{-u}B \cap M_{1/u}B ) = \lim_{N\to \infty} \frac{1}{|F_N|} \sum_{u\in F_N} \int_X M_u \mathbbm{1}_B \cdot M_u A_{-u}\mathbbm{1}_B \cdot \mathbbm{1}_B\ d\mu,    
\end{align*}
as in the proof of Corollary \ref{main ergodic theorem corollary proof}.
By Theorem \ref{P_Af} for $f=g=\mathbbm{1}_B$, this limit becomes 
\begin{equation} \label{72}
\int_X P_A\mathbbm{1}_B \cdot P_M\mathbbm{1}_B \cdot \mathbbm{1}_B\ d\mu = P_A\mathbbm{1}_B \int_X P_M\mathbbm{1}_B \cdot \mathbbm{1}_B\ d\mu \geq (\mu(B))^3,
\end{equation} 
because $P_A\mathbbm{1}_B=\mu(B)$, $P_M$ is an orthogonal projection and $P_M1=1$. 

For the second part, if in addition $S_M$ acts ergodically, then 
$P_M\mathbbm{1}_B=\mu(B)$ and the same method gives the result.
\end{proof}

\begin{center}
\section{Generalization of Shkredov's theorem} \label{Shkredov}
\end{center}

This section is devoted to the proof of Theorem 
\ref{main theorem Shkredov}, which generalizes a 
result due to Shkredov pertaining to finite fields of prime 
order, as mentioned in Section \ref{main results}. We actually prove the following slightly more general theorem.

\begin{theorem} \label{main theorem F*} 
Let $F$ be any finite field. Let also $B_1,B_2,B_3\subset F^{*}$ be any sets satisfying $|B_1||B_2||B_3|>7|F|^{5/2}$. Then, there exists $u,v\in F^{*}$ such that $v\in B_1, u+v \in B_2$ and $uv\in B_3$.
\end{theorem}

We have stated Theorem \ref{main theorem F*} for 
subsets of $F^{*}$ because working with an 
indicator function $g=\mathbbm{1}_B$ of a set 
$B\subset F^{*}$ allows us to use 
inequalities like $\mu(B) \leq P_Mg(x) \leq (|F|/|F^{*}|) \mu(B)$, for 
all $x\neq 0$, which simplifies the proof. However, we 
do not lose generality as our main result, Theorem 
\ref{main theorem Shkredov}, is an immediate corollary of 
Theorem \ref{main theorem F*}.

\begin{proof}[Proof that Theorem \ref{main theorem F*} implies Theorem \ref{main theorem Shkredov}]
Let $B_1,B_2,B_3\subset F$ be any sets satisfying 
$|B_1||B_2||B_3|>8|F|^{5/2}$
and let $B'_i=B_i \cap F^{*} \subset F^{*}$, for $i=1,2,3$. Then, 
$$|B'_1||B'_2||B'_3| \geq (|B_1|-1)(|B_2|-1)(|B_3|-1)$$
and the right hand side is larger than
$$|B_1||B_2||B_3|-|B_1||B_2|-|B_1||B_3|-|B_2||B_3| \geq |B_1||B_2||B_3|-3|F|^2 > 7|F|^{5/2},$$
where the last inequality holds because 
$3|F|^2 \leq |F|^{5/2}$, for any field of order at least $9$. Then the result follows by an application of Theorem \ref{main theorem F*} for the sets $B'_1,B'_2,B'_3$.
\end{proof}

We now proceed to prove Theorem \ref{main theorem F*}. 
This proof is an effort to a ``finitise'' the proof of 
Theorem \ref{main ergodic theorem 2}. 
However, there are some additional technicalities here, because 
quantities that 
vanish in the infinite setting are replaced by 
``error'' terms which are 
bounded (and go to $0$ asymptotically as $|F|$ increases to $\infty$). 

As in the infinite setting, the proof of Theorem 
\ref{main theorem F*} relies on a finitistic version of the double 
ergodic theorem of Theorem \ref{P_Af}, which is stated in 
Proposition \ref{L^2 bound for Shkredov} below. In order to ease 
the discussion, we first prove the 
following estimate that appears in the proof of the latter.

\begin{proposition} \label{L^2 bound for Shkredov general 2}
Let $F$ be any finite field and $f=\mathbbm{1}_{B}-\mu(B)$ for some $B \subset F^{*}$. Then, 
$$\frac{1}{|F^{*}|} \sum_{v\in F^{*}} \norm{\frac{1}{|F^{*}|}\sum_{u\in F} M_{v}A_{-uv}f \cdot  A_{-u}f}^2_2 \leq \frac{6}{|F|} \norm{f}^4_2.$$    
\end{proposition}

\begin{proof}
By Proposition \ref{finite vdC} we have that for any $v\in F^{*}$
\begin{equation*} 
\norm{\sum_{u\in F} M_{v}A_{-uv}f \cdot  A_{-
u}f}^2_2=\sum_{u,w\in F} \langle M_{v}A_{-(u+w)v}f 
\cdot  A_{-(u+w)}f\ ,\ M_{v}A_{-uv}f \cdot  A_{-
u}f\rangle.
\end{equation*}
Now, as $M_vA_{-(u+w)v}=A_{-(u+w)v^2}M_v$ and $M_vA_{-uv}=A_{-uv^2}M_v$ by \eqref{M_uA_v} and $A_{uv^2}$ preserves $\mu$, we see that
\begin{equation*} 
\norm{\sum_{u\in F} M_{v}A_{-uv}f \cdot  A_{-u}f}^2_2=\sum_{u,w\in F} \langle A_{-wv^2}M_vf \cdot  A_{u(v^2-1)-w}f\ ,\ M_vf \cdot  A_{u(v^2-1)}f\rangle.    
\end{equation*}
Observe that we can rewrite this as
\begin{equation} \label{eq:35}
\norm{\sum_{u\in F} M_{v}A_{-uv}f \cdot  A_{-u}f}^2_2 = 
\sum_{u,w\in F} \langle A_{u(v^2-1)} \left(f \cdot A_{-w}f\right)\ 
,\ M_v \left(f \cdot A_{-wv} f\right) \rangle.    
\end{equation}
Whenever $v^2\neq 1$ we have that 
\begin{align}  \label{eq:34} 
\sum_{u,w\in F} \langle A_{u(v^2-1)} \left(f \cdot A_{-w}f\right)\ ,\ M_v \left(f \cdot A_{-wv} f\right) \rangle \nonumber & = \\
\sum_{w\in F} \langle |F| \cdot P_A\left(f \cdot A_{-w}f\right)\ ,\ M_v \left(f \cdot A_{-wv} f\right) \rangle \nonumber & =  & \quad \text{by definition of $P_A$}  \\ 
\sum_{w\in F} |F| \cdot \int_X f \cdot A_{-w}f\ d\mu \int_X M_v \left(f \cdot A_{-wv} f\right) \ d\mu \nonumber & =  & \quad \text{by ergodicity of $S_A$} \\
\sum_{w\in F} |F| \cdot \int_X f \cdot A_{-w}f\ d\mu \int_X f \cdot A_{-wv} f \ d\mu & .  & \quad \text{by invariance of $M_v$}.
\end{align}
Using \eqref{eq:34} in \eqref{eq:35} we see that 
\begin{multline} \label{eq:36}
\frac{1}{|F^{*}|} \sum_{v\in F^{*}} \norm{\frac{1}{|F^{*}|} \sum_{u\in F} M_{v}A_{-uv}f \cdot  A_{-u}f}^2_2= \\
\frac{|F|}{|F^{*}|^3} \sum_{v\notin \{0,1,-1\}}\sum_{w\in F} \int_X f \cdot A_{-w}f\ d\mu \int_X f \cdot A_{-wv} f \ d\mu \ + \\
\frac{|F|}{|F^{*}|^3}\sum_{w\in F}\left(\langle f\cdot A_{-w}f\ ,\ f\cdot A_{-w}f + M_{-1}\left( f\cdot A_{w}f \right) \rangle \right). 
\end{multline}
Moreover, 
\begin{equation} \label{eq:36'}
\sum_{w\in F} \langle f\cdot A_{-w}f\ ,\ f\cdot A_{-w}f \rangle = \langle f^2 \ , \ \sum_{w\in F}A_{-w}f^2 \rangle = |F| \cdot \norm{f}^4_2
\end{equation}
and similarly, 
\begin{multline} \label{eq:36''}
\sum_{w\in F} \langle f\cdot A_{-w}f\ ,\ M_{-1}(f\cdot A_{w}f) \rangle = \langle f\cdot M_{-1}f \ , \ \sum_{w\in F} A_{-w} (f\cdot M_{-1}f) \rangle \leq |F| \cdot \norm{f}^4_2.
\end{multline}
Now, for each $w\neq 0$, we have that
$$\sum_{v\in F} \int_X f \cdot A_{-w}f\ d\mu \int_X f \cdot A_{-wv} f \ d\mu = \int_X f \cdot A_{-w}f\ d\mu \int_X f \cdot P_Af \ d\mu =0 $$
and so 
$$\sum_{v\in F}\sum_{w\in F} \int_X f \cdot A_{-w}f\ d\mu \int_X f \cdot A_{-wv} f \ d\mu=\sum_{v\in F} \left(\int_X f^2\ d\mu\right)^2=|F|\cdot \norm{f}^4_2.$$
Therefore, 
\begin{multline} \label{eq:37}
\sum_{v\notin \{0,1,-1\}}\sum_{w\in F} \int_X f \cdot A_{-w}f\ d\mu \int_X f \cdot A_{-wv} f \ d\mu =\\
|F|\cdot \norm{f}^4_2 \ - \sum_{v\in \{0,1,-1\}}\sum_{w\in F} \int_X f \cdot A_{-w}f\ d\mu \int_X f \cdot A_{-wv} f \ d\mu \leq 2\cdot |F|\cdot \norm{f}^4_2.
\end{multline}
The last inequality follows because the rightmost sum 
vanishes for $v=0$ and is non-negative when $v=1$. 
In view of \eqref{eq:37}, the equality in 
\eqref{eq:36} is replaced by  
$$\frac{1}{|F^{*}|} \sum_{v\in F^{*}} \norm{\frac{1}{|F^{*}|} \sum_{u\in F} M_{v}A_{-uv}f \cdot  A_{-u}f}^2_2 \leq 2\frac{ |F|^2}{|F^{*}|^3} \norm{f}^4_2 + 2\frac{|F|^2}{|F^{*}|^3}\norm{f}^4_2 \leq \frac{6}{|F|} \norm{f}^4_2,$$
where in the first inequality we also used 
\eqref{eq:36'} and \eqref{eq:36''} and the last 
inequality holds whenever $|F|\geq 8$.
\end{proof}

We now prove Proposition \ref{L^2 bound for Shkredov}.

\begin{proposition} \label{L^2 bound for Shkredov}
Let $F$ be any finite field and let $f=\mathbbm{1}_{B}-\mu(B)$ 
for some $B \subset F^{*}$ and $g=\mathbbm{1}_C$, 
for some $C \subset F^{*}$. Then
\begin{equation} \label{eq:30} 
\norm{\frac{1}{|F^{*}|} \sum_{u\in F^{*}} M_uA_{-u}f \cdot M_u g }_2^2 \leq \\ 
\frac{7}{\sqrt{|F|}}\mu(B)\mu(C).
\end{equation}
\end{proposition}    

\begin{proof}
By Proposition \ref{finite vdC} and the fact that $M_u$ 
preserves $\mu$ for all $u\in F^{*}$ we see that
$$\norm{\frac{1}{|F^{*}|} \sum_{u\in F^{*}} M_uA_{-u}f \cdot M_u g }_2^2 = \frac{1}{|F^{*}|^2} \sum_{u,v\in F^{*}} \langle M_{v}A_{-uv}f \cdot M_vg\ ,\ A_{-u}f \cdot g \rangle.$$
As all functions are real-valued, the above can be rewritten 
as
$$ \langle g\ ,\  \frac{1}{|F^{*}|} \sum_{v\in F^{*}} M_vg  \cdot \left( \frac{1}{|F^{*}|} \sum_{u\in F^{*}} M_{v}A_{-uv}f \cdot  A_{-u}f \right) \rangle.$$
Hence, using the Cauchy-Schwarz inequality we see that 
\begin{equation} \label{eq:31}
\norm{\frac{1}{|F^{*}|} \sum_{u\in F^{*}} M_uA_{-u}f \cdot M_u g }_2^2 \leq \norm{g}_2 \norm{\frac{1}{|F^{*}|} \sum_{v\in F^{*}} M_vg  \cdot \left( \frac{1}{|F^{*}|} \sum_{u\in F^{*}} M_{v}A_{-uv}f \cdot  A_{-u}f \right)}_2.
\end{equation}
By the triangle inequality, the right hand side in \eqref{eq:31} is less than or equal to 
\begin{equation*}
\norm{g}_2 \norm{\frac{1}{|F^{*}|} \sum_{v\in F^{*}} M_vg  \cdot \left( \frac{1}{|F^{*}|} \sum_{u\in F} M_{v}A_{-uv}f \cdot  A_{-u}f \right)}_2+\norm{g}_2 \norm{\frac{1}{|F^{*}|^2} \sum_{v\in F^{*}} M_vg  \cdot M_vf \cdot  f}_2 
\end{equation*}
and then 
$$\norm{g}_2 \norm{\frac{1}{|F^{*}|^2} \sum_{v\in F^{*}} M_vg  \cdot M_vf \cdot  f}_2=\frac{1}{|F^{*}|}\norm{g}_2\norm{P_M(f\cdot g)\cdot f}_2\leq \frac{|F|}{|F^{*}|^2}\norm{g}^2_2\norm{f}^2_2, $$
as $g(0)=0$ and so $P_M(f\cdot g) \leq \left(|F|/|F^{*}|\right) \langle f , g\rangle \leq \left(|F|/|F^{*}|\right) \norm{f} \norm{g}$, by the comments after Theorem \ref{main theorem F*} and the Cauchy-Schwarz inequality. Therefore, 
\begin{multline} \label{eq:33}
\norm{\frac{1}{|F^{*}|} \sum_{u\in F^{*}} M_uA_{-u}f \cdot M_u g }_2^2 \leq \\ 
\norm{g}_2 \norm{\frac{1}{|F^{*}|} \sum_{v\in F^{*}} M_vg \cdot \left( \frac{1}{|F^{*}|} \sum_{u\in F} M_{v}A_{-uv}f \cdot  A_{-u}f \right)}_2 +  \frac{|F|}{|F^{*}|^2}\norm{g}^2_2\norm{f}^2_2. \   
\end{multline}
By an application of Cauchy-Schwarz's inequality for sums of 
products we have that
\begin{align} 
\norm{\frac{1}{|F^{*}|} \sum_{v\in F^{*}} M_vg  \cdot \left( \frac{1}{|F^{*}|} \sum_{u\in F} M_{v}A_{-uv}f \cdot  A_{-u}f \right)}^2_2 \nonumber  & \leq \\
\int_X \frac{1}{|F^{*}|} \sum_{v\in F^{*}} (M_vg)^2 \cdot \frac{1}{|F^{*}|} \sum_{v\in F^{*}} \left( \frac{1}{|F^{*}|} \sum_{u\in F} M_{v}A_{-uv}f \cdot  A_{-u}f \right)^2\ d\mu \nonumber & = \\
\int_X P_Mg \cdot \frac{1}{|F^{*}|} \sum_{v\in F^{*}}  \left( \frac{1}{|F^{*}|}\sum_{u\in F} M_{v}A_{-uv}f \cdot  A_{-u}f\right)^2\ d\mu \nonumber & \leq  \\
\frac{|F|}{|F^{*}|}\mu(C)\ \cdot \frac{1}{|F^{*}|} \sum_{v\in F^{*}} \norm{\frac{1}{|F^{*}|} \sum_{u\in F} M_{v}A_{-uv}f \cdot  A_{-u}f}^2_2 &  \  \label{eq:32}.
\end{align}
By Proposition \ref{L^2 bound for Shkredov general 2} we see that 
$$\frac{1}{|F^{*}|} \sum_{v\in F^{*}} \norm{\frac{1}{|F^{*}|} \sum_{u\in F} M_{v}A_{-uv}f \cdot  A_{-u}f}^2_2 \leq \frac{6}{|F|} \norm{f}^4_2.$$
Using this in \eqref{eq:32} and the bound in 
\eqref{eq:33} we have that
\begin{equation*}
\norm{\frac{1}{|F^{*}|} \sum_{u\in F^{*}} M_uA_{-u}f \cdot M_u g }_2^2 \leq \frac{\sqrt{6}}{\sqrt{|F|}}\frac{\sqrt{|F|}}{\sqrt{|F^{*}|}}\norm{g}^2_2\norm{f}^2_2
+ \frac{|F|}{|F^{*}|^2}\norm{g}^2_2 \norm{f}^2_2 \leq \frac{\sqrt{6}+1}{\sqrt{|F^{*}|}}\norm{g}^2_2\norm{f}^2_2
\end{equation*}
Finally, it follows by the definition of $f$ that $\norm{f}^2_2 \leq 2\mu(B)$, as shown in the proof of Theorem $5.1$ in \cite{2}. In conclusion, \eqref{eq:33} becomes 
$$\norm{\frac{1}{|F^{*}|} \sum_{u\in F^{*}} M_uA_{-u}f \cdot M_u g }_2^2 \leq \frac{8}{\sqrt{|F|}}\mu(B)\mu(C),$$
since $2(\sqrt{6}+1) \sqrt{|F| \big/ |F^{*}|} \leq 8,$ whenever $|F|\geq 8$. 
\end{proof}

We are finally in the position to prove the main result of this 
section, Theorem \ref{main theorem F*}.

\begin{proof}[Proof of Theorem \ref{main theorem F*}]

Using the same notation as in 
Section \ref{finite fields p(x) section}, the
assumption of Theorem \ref{main theorem F*} can be
rewritten as $\mu(B_1)\mu(B_2)\mu(B_3)> 7 / \sqrt{|F|}$ and its conclusion is equivalent to the existence of $u\in F^{*}$ so that
$\mu(B_1 \cap A_{-u}B_2 \cap M_{1/u}B_3)>0,$
where $\mu$ is the normalised counting measure on $F$. It will thus suffice to show that 
$\sum_{u \in F^{*}} \mu(B_1 \cap A_{-u}B_2 \cap M_{1/u}B_3)>0.$
Using the fact that $M_u$ preserves $\mu$ for all $u\in F^{*}$, this is equivalent to  
\begin{equation} \label{eq:26}
\langle  \mathbbm{1}_{B_3} \ , \  \frac{1}{|F^{*}|} \sum_{u\in F^{*}} M_u A_{-u}\mathbbm{1}_{B_2} \cdot M_u\mathbbm{1}_{B_1} \rangle > 0.
\end{equation}
We let $f=\mathbbm{1}_{B_2} -P_A\mathbbm{1}_{B_2}$. 
Observe that $P_Af=0$ and 
$P_A\mathbbm{1}_{B_2}=\mu(B_2)$ is a constant. Then,
\begin{multline} \label{eq:27}
\langle \mathbbm{1}_{B_3}\ ,\ \frac{1}{|F^{*}|} \sum_{u\in F^{*}} M_uA_{-u}\mathbbm{1}_{B_2} \cdot M_u \mathbbm{1}_{B_1} \rangle = \\
\mu(B_2) \langle \mathbbm{1}_{B_3}\ ,\ \frac{1}{|F^{*}|} \sum_{u\in F^{*}} M_u \mathbbm{1}_{B_1} \rangle + \langle \mathbbm{1}_{B_3}\ ,\ \frac{1}{|F^{*}|} \sum_{u\in F^{*}} M_uA_{-u}f \cdot M_u\mathbbm{1}_{B_1} \rangle =\\
\mu(B_2) \langle \mathbbm{1}_{B_3}\ , P_M\mathbbm{1}_{B_1}  \rangle + \langle \mathbbm{1}_{B_3}\ ,\ \frac{1}{|F^{*}|} \sum_{u\in F^{*}} M_uA_{-u}f \cdot M_u\mathbbm{1}_{B_1} \rangle.
\end{multline}
As $B_1 \subset F^{*}$ it follows by the comments after Theorem \ref{main theorem F*} that
$$\mu(B_2) \langle \mathbbm{1}_{B_3}\ , P_M\mathbbm{1}_{B_1}  \rangle \geq \mu(B_1)\mu(B_2)\mu(B_3).$$
Using this in \eqref{eq:27}, we reduce \eqref{eq:26} to showing that 
\begin{equation*} 
\left| \langle \mathbbm{1}_{B_3}\ ,\ \frac{1}{|F^{*}|} \sum_{u\in F^{*}} M_uA_{-u}f \cdot M_u \mathbbm{1}_{B_1} \rangle \right| < \mu(B_1)\mu(B_2)\mu(B_3).
\end{equation*}
Applying the Cauchy-Schwarz inequality the latter follows from showing that
\begin{equation} \label{eq:28}
\norm{\mathbbm{1}_{B_3}} \norm{\frac{1}{|F^{*}|} \sum_{u\in F^{*}} M_uA_{-u}f \cdot M_u \mathbbm{1}_{B_1}}_2 < \mu(B_1)\mu(B_2)\mu(B_3).
\end{equation}
In Proposition \ref{L^2 bound for Shkredov} we 
showed that
$$\norm{\frac{1}{|F^{*}|} \sum_{u\in F^{*}} M_uA_{-u}f \cdot M_u \mathbbm{1}_{B_1} }_2^2 \leq \frac{7}{\sqrt{|F|}}\mu(B_1)\mu(B_2)$$
and since $\norm{\mathbbm{1}_{B_3}}=\sqrt{\mu(B_3)}$, 
we see that
\eqref{eq:28} holds whenever 
$$\frac{\sqrt{7}}{|F|^{1/4}}\sqrt{\mu(B_1)\mu(B_2)\mu(B_3)} < \mu(B_1)\mu(B_2)\mu(B_3),$$
which is equivalent to our main assumption, namely that $7 \big/ \sqrt{|F|} < \mu(B_1)\mu(B_2)\mu(B_3)$.
\end{proof}

As a corollary of the proof we get the following 
quantitative result.

\begin{corollary} \label{corollary}
Let $F$ be any finite field. Let also $B_1,B_2,B_3\subset F^{*}$ be any sets satisfying $|B_1||B_2||B_3|>7|F|^{5/2}$. Then, for each $s<\ell:=\min{\{|B_1|,|B_2|,|B_3|\}}$ there is a set $D\subset F^{*}$ of cardinality 
$$ |D| \geq  \frac{|B_1||B_2||B_3||F^{*}| \big/ |F|^2-\sqrt{7|B_1||B_2||B_3||F^{*}|^2 \big/ |F|^{3/2}}-s|F^{*}|}{\ell},$$
so that for each $u\in D$ there are $s$ choices for $v\in F$ such that $v\in B_1$, $u+v\in B_2$ and $uv\in B_3$.
\end{corollary}

\begin{proof}
Let $\delta=s/|F|$ for any $s$ as above and let 
$$D=\{u\in F^{*}: \mu(B_3 \cap M_uA_{-u}B_2 \cap M_uB_1)>\delta\}.$$ 

Similarly to the proof of Corollary \ref{estimates 1}, it follows 
from the proof of Theorem \ref{main theorem F*} that 
\begin{equation} \label{eq:23}
\frac{|D|}{|F^{*}|} \geq  \frac{\mu(B_1)\mu(B_2)\mu(B_3)-\sqrt{7\mu(B_1)\mu(B_2)\mu(B_3) \big/ |F|^{1/2}}-\delta}{m}, 
\end{equation}
where $m : = \min{\{ \mu(B_1), \mu(B_2), \mu(B_3) \}}$. By the definition of $\mu$, \eqref{eq:23} is equivalent to
\begin{equation}\label{eq:24}
|D| \geq  \frac{|B_1||B_2||B_3||F^{*}| \big/ |F|^2-\sqrt{7|B_1||B_2||B_3||F^{*}|^2 \big/ |F|^{3/2}}-s|F^{*}|}{\ell}.
\end{equation}
Finally, we see that for each $u\in D$,
$$\frac{s}{|F|} \leq \mu(B_3 \cap M_uA_{-u}B_2 \cap M_uB_1) = \mu(M_{1/u}B_3 \cap A_{-u}B_2 \cap B_1) = \frac{\left|M_{1/u}B_3 \cap A_{-u}B_2 \cap B_1\right|}{|F|}$$
and thus there are $s$ choices for $v \in F$ 
satisfying $v\in B_1, v+u\in B_2$ and $vu\in B_3$.
\end{proof}

\begin{remark}
The proof 
of Corollary \ref{corollary} shows in particular that if 
$A\subset F$ 
satisfies $|A|\geq \alpha|F|$, for some $\alpha \in 
(0,1)$, then $|D|\geq c_{\alpha}|F|$, for some constant 
$c_{\alpha}>0$ that does not depend on $F$. This follows 
by taking $B_1=B_2=B_3=A$ above and choosing $s=\alpha'|F|$ for 
some $\alpha' < \alpha$ and $n\in \NN$ large enough so that the 
right hand 
side in \eqref{eq:24} is positive whenever $|F|>n$.  Thus, there 
are $s |D| \geq c'_{\alpha}|F|^2$ triples 
$\{v,v+u,vu\} \subset A$, where $c'_{\alpha}>0$ is another 
constant that does not depend on $|F|$.
\end{remark}

\begin{center}
\section{A conditional generalisation of Green and 
Sanders' theorem}    \label{Green-Sanders}
\end{center}

In Section \ref{partition for finite fields} we 
devised a 
finitistic ``colouring 
trick'' to prove Theorem \ref{finite fields p(x) partition version} 
from Corollary 
\ref{estimates 1}. Now, using a similar argument 
and a finitistic version of 
Conjecture \ref{conjecute countable fields} as our 
basis we will prove a generalisation of 
Green and Sanders' theorem about ``monochromatic sums 
and products'' in finite fields as mentioned in the 
introduction.  

Before stating the aforementioned conjecture, we make another 
related 
conjecture that would generalise a special case of Theorem \ref{main theorem F*}. 

\begin{conjecture} \label{Conjecture finite fields}
Let $F$ be any finite field and assume that 
$\mathcal{A}_{F}$ acts by m.p.t. on a probability space 
$(X, \mathcal{X}, \nu)$. Let $B\in \mathcal{X}$ be a 
set with 
$\nu(B) > \left( c\big/ |F|\right)^{a}$, for some constants $a,c>0$. Then, there exists $u\in F^{*}$ such that 
$$\nu(B\cap A_{-u}B\cap M_{1/u}B)>0.$$    
\end{conjecture}

\begin{remark*}
Observe that when $X=F$ and $\nu=\mu$, the counting 
measure on $F$, Theorem \ref{main theorem F*} with $B_1=B_2=B_3$ 
is a 
special of this conjecture with $a=1/6$. However, for this special
case we knew that the additive action of $S_A$ is ergodic, which 
seems to have been heavily used in the proof of Theorem 
\ref{main theorem F*}, and is no longer true in the general case.
\end{remark*}

For the purpose of proving the generalisation of Green and Sanders' 
theorem, that is, Conjecture \ref{G-S statement}, we actually need 
only consider a special case of Conjecture 
\ref{Conjecture finite fields} with $X=F^m$ and 
$\nu=\mu^m$, some $m\in \NN$, where $\mu$ is the counting 
measure on $F$, and $B=B_1 \times \dots \times B_m \subset F^m$ is
a set with $\nu(B) > \left( c\big/ |F|\right)^{a}$, for some 
constants $a,c>0$. 

A way one could try to prove the aforementioned special case of 
Conjecture 
\ref{Conjecture finite fields} would start by decomposing $g=\mathbbm{1}_B$ as $P_Ag+f$, 
where $f=g-P_Ag$. Then, following Section \ref{Shkredov} and 
considering the inner product
$\langle f , g\rangle = \frac{1}{|F^m|} \sum_{x\in F^m} f(x)\cdot 
\overline{g(x)}$, one would have to show that 
\begin{equation} \label{73}
\frac{1}{|F|}\sum_{u\in F^{*}} \langle g ,  M_uA_{-u}P_Ag\cdot M_ug \rangle + \frac{1}{|F|}\sum_{u\in F^{*}} \langle g , M_uA_{-u}f\cdot M_ug \rangle > 0. 
\end{equation}
This time $P_Ag$ is not necessarily a constant, however we still have that 
$$\frac{1}{|F|}\sum_{u\in F^{*}} \langle g , M_uA_{-u}P_Ag\cdot M_ug \rangle = \langle g ,  P_M(P_Ag\cdot g) \rangle \geq (\nu(B))^4.$$
Indeed, as $P_Ag \leq 1$ and $P_M$ is an orthogonal projection with $P_M1=1$ we have 
$$\langle g , P_M(P_Ag\cdot g) \rangle \geq \langle 
P_Ag\cdot g ,  P_M(P_Ag\cdot g) \rangle = 
\norm{P_M(P_Ag\cdot g)}^2_2 \geq \left( \int_{F^m} 
P_Ag\cdot g\ d\nu \right)^2, $$
where the last inequality is Cauchy-Schwarz. Then, arguing similarly for $P_A$ we have
$$\left( \int_{F^m} 
P_Ag\cdot g\ d\nu \right)^2 \geq \left( \int_{F^m} g\ d\nu \right)^4=(\nu(B))^4.$$
Therefore, the proof would follow from the following statement, 
which is precisely what we are going to use.

\begin{conjecture} \label{quantitative Conjecture for G-S}
Let $F$ be any finite field and let $m\in \NN$. 
Consider the coordinate-wise affine action of $\mathcal{A}_{F}$ by m.p.t. on $(F^m, \nu)$, where $\nu=\mu^m=\mu \times \dots \times \mu$. Let $f=\mathbbm{1}_B-P_A(\mathbbm{1}_B)$, where $B=B_1 \times \dots \times B_m \subset F^m$ and $g=\mathbbm{1}_B$. Then, 
$$\norm{\frac{1}{|F^{*}|} \sum_{u\in F^{*}} M_uA_{-u}f \cdot M_ug}_2 \leq \frac{c}{|F|^b}\norm{f}_2\norm{g}_2,$$
for some $b,c>0$. 
\end{conjecture}

As a corollary of Conjecture \ref{quantitative Conjecture for G-S} 
we get the following estimates on the set of return times in the 
special case of Conjecture \ref{Conjecture finite fields} that we 
need. The (conditional) proof is a straightforward adjustment of 
the proof of Corollary \ref{corollary} and so we omit it.

\begin{conjecture} \label{estimates'}
Let $F$ be a finite field and $m\in \NN$. Assume that $\mathcal{A}_{F}$ acts on $(F^m,\nu)$ by m.p.t. as above. Let $B=B_1\times \dots \times B_m \subset F^m$ and $\delta < \nu(B)$. 
Then, the set 
$$D:=\{u\in F^{*}: \nu(B \cap A_{-u}B \cap M_{1/u}B)>\delta\},$$ 
satisfies
\begin{equation} \label{6'}
\frac{|D|}{|F^{*}|} \geq \frac{(\nu(B))^4 - c\cdot (\nu(B))^{3/2} \big/ |F|^{b} - \delta}{\nu(B)}.    
\end{equation}    
\end{conjecture}

We are now in a position to apply a version of the 
finitary ``colouring trick'' and recover Conjecture \ref{G-S statement}, which we recall for convenience.

\begin{theorem***}
Let $r\in \NN$ be a number of colours. Then, there is $n(r) \in \NN$, so that 
for any finite field $F$ with $|F|\geq n(r)$, any colouring 
$F=C_1 \cup \cdots \cup C_r$ contains $ d_r|F|^2$ monochromatic quadruples 
$\{u,v,u+v,uv\}$, where $d_r>0$ is some constant that does not depend on $|F|$.      
\end{theorem***}

\begin{remark}
Setting $d_r'=d_r/r$ we get a colour class containing at least 
$d_r'|F|^2$ monochromatic patterns of the form $\{u,v,u+v,uv\}$. 
Moreover, the proof gives an upper bound smaller than 
$n(r)=r^{4^{(r+2)}}$ for the 
$r$-Ramsey number for monochromatic patterns $\{u,v,u+v,uv\}$ in 
this setting. That
is, this conditional proof guarantees that for any $r$-colouring of 
a finite field $F$ with $|F| \geq r^{4^{(r+2)}}$, one of the 
colours must contain a non-trivial quadruple $\{u,v,u+v,uv\}$.
\end{remark}

\begin{proof}
Let $r\in \NN$, $r>1$, be fixed and let $F$ be any finite field with $|F| \geq n(r)$, for $n(r)$ to be determined later. For an $r$-colouring of such a field we can permute the colours if necessary and assume that $|C_1| \geq |C_2| \geq \dots \geq |C_r|$. Clearly, then, $|C_1| \geq |F|\big/ r$. Next, we pick a number $1\leq r' \leq r$ in the following manner. If $|C_2| < |F|\big/r^{16}$, we set $r'=1$. Else, we have that $|C_2| \geq |F|\big/r^{16}$ and $r'\geq 2$. 
Then, we either have that $|C_3| \geq |F|\big/r^{64}$, whence $r'\geq 2$ or not and 
let $r'=2$. Proceeding in this fashion we set 
$$r':= \max{\Big\{1 \leq j \leq r: |C_1| \geq |F|\big/r\ , \ |C_2| \geq |F|\big/r^{16}\ ,\ \dots \ , \ |C_{j}| \geq |F|\big/r^{4^{j}} \Big\}}.$$ 
Let $C=C_1 \times \dots \times C_{r'}$. We consider the natural measure 
preserving action of $\mathcal{A}_{F}$ on $F^{r'}$ (defined coordinate-wise), 
with the counting measure $\nu$ given by $\nu(E)=|E|/|F^{r'}|$, for 
any 
$E \subset F^{r'}$. So, for $C_1,\dots,C_{r'} \subset F$ we have 
that 
$\nu(C_1 \times \cdots \times C_{r'})=\mu(C_1)\cdots \mu(C_{r'})$, where $\mu$ is the 
normalised counting measure on $F$. For any $\delta :=s\big/ |F^*| < \nu(C)$ let 
\begin{equation*} 
D=\{u\in F^{*}: \nu(C \cap A_{-u}C \cap M_{1/u}C) > \delta \}.
\end{equation*}
Then, by Corollary \ref{estimates'} we have that
\begin{equation*} 
\frac{|D|}{|F^{*}|} \geq \frac{(\nu(C))^4 - c\cdot (\nu(C))^{3/2} \big/ |F^{*}|^{b} - \delta}{\nu(C)},  
\end{equation*} 
which implies that
\begin{equation} \label{59}  
|D| \geq (\nu(C))^3|F^{*}| - c\cdot |F^{*}|^{1-b} - \frac{|F^{*}|\delta}{\nu(C)}.
\end{equation}
We want to bound below the size of  $D \setminus \left( C_{r'+1} \cup \cdots \cup C_r \right)$,
because, for any element $u$ in this set, it holds that $u\in C_1 \cup \cdots \cup C_{r'}$ and also that $\nu(C \cap A_{-u}C \cap M_{1/u}C) > \delta$. Then, if $u \in C_j$,  for $1 \leq j \leq r'$, by the definition of $C$ and the measure $\nu$ we have that $\mu(C_j \cap A_{-u} C_j \cap M_{1/u}C) > \delta$ and hence $|C_j \cap C_j/u \cap (C_j-u)| > s$, which implies the existence of at least $s-$elements 
$v\in F^{*}$ such that $\{u,v,u+v,uv\} \subset C_j$. To this end, 
by the choice of $r'$ we have
\begin{equation} \label{61}
|C_{r'+1}|+\dots+|C_r| \leq  (r-r')|F|\big/ r^{4^{(r'+1)}} < |F|\big/ r^{4^{(r'+1)}-1}.
\end{equation} 
Using the definition of $C$ and $r'$ it holds that 
\begin{equation} \label{63}
\nu(C)=\frac{|C_1|\cdots |C_{r'}|}{|F^{r'}|} \geq \frac{1}{r} \cdot \frac{1}{r^{16}}\cdot \frac{1}{r^{64}} \cdots \frac{1}{r^{4^{r'}}}=\frac{1}{r^{(1+16+64+\dots+4^{r'})} }.    
\end{equation}
Now, $$\left| D \setminus \left( C_{r'+1} \cup \cdots \cup C_r \right) \right| \geq |D|-\left(|C_{r'+1}|+\dots + |C_r| \right)$$
and so by \eqref{59}, \eqref{61} and \eqref{63} we see that 
\begin{equation} \label{62}
\left| D \setminus \left( C_{r'+1} \cup \cdots \cup C_r \right) \right| \geq 
|F^{*}|\big/r^{3(1+16+64+\dots+4^{r'})} - c\cdot |F^{*}|^{1-b} - \frac{|F^{*}|\delta}{\nu(C)}- |F|\big/ r^{4^{(r'+1)}-1}.  
\end{equation}
The quantity at the right hand side of \eqref{62} can be rewritten as 
$$|F^{*}| \left( 1\big/r^{3(1+16+64+\dots+4^{r'})}-1\big/ r^{4^{(r'+1)}-1}-\delta\big/ \nu(C) \right)- c\cdot |F^{*}|^{1-b}-1\big/ r^{4^{(r'+1)}-1}.$$
Now, one can see that\footnote{For $r'\geq 2$ we have that 
$4^{(r'+1)}-1-3\left(4^{r'}+\dots+4^2+1 \right)=12$}
$$\frac{1}{r^{3(1+16+\dots+4^{r'})} }-\frac{1}{r^{4^{(r'+1)}-1}} = \frac{r^{4^{(r'+1)}-1-3(4^{r'}+\dots+4^2+1)}-1}{r^{4^{(r'+1)}-1}} = \frac{r^{12}-1}{r^{4^{(r'+1)}-1}}.$$
Therefore, the right hand side of \eqref{62} is greater than or equal to
\begin{equation}\label{67}
|F^{*}| \left( \frac{r^{12}-1}{r^{4^{(r'+1)}-1}}-\delta \cdot r^{(1+16+\dots+4^{r'})} \right)- c\cdot |F^{*}|^{1-b}-1\big/ r^{4^{(r'+1)}-1}=c_r \cdot |F^{*}|,
\end{equation} 
which follows by setting
$$c_r =  \frac{r^{12}-1}{r^{4^{(r'+1)}-1}}-\delta \cdot r^{(1+16+\dots+4^{r'})} - c\big/ |F^{*}|^b -1\big/ \left( |F^{*}|r^{4^{(r'+1)}-1} \right).$$
Recall that $|F| \geq n(r)$. We choose $n(r)$ large enough to 
guarantee that $c_r>0$. Since $\delta=s\big/ |F^{*}|$ and for any 
$u \in  D \setminus \left( C_{r'+1} \cup \cdots \cup C_r \right)$ we have at least $s$ 
monochromatic quadruples $\{u,v,u+v,uv\}$, it follows by \eqref{67} that there are in total 
at least 
$$s \cdot c_r \cdot |F^{*}| =\delta \cdot c_r \cdot |F^{*}|^2 = d_r |F|^2\ $$
monochromatic patterns of the form $\{u,v,u+v,uv\}$, where $d_r>0$ is a constant that does 
not depend on the size of $F$. 
\end{proof}

\medskip

\end{document}